\documentclass[final,leqno]{siamltex}

\usepackage{diagbox}
\usepackage{multirow}
\usepackage{amsfonts,amssymb} 
\usepackage{amsmath} 
\usepackage{color} 
\usepackage{graphicx} 
\usepackage{epsfig,mathrsfs} 
\usepackage[bf,SL,BF]{subfigure} 
\usepackage{fancyhdr} 
\usepackage{CJK} 
\usepackage{caption} 
\usepackage{wrapfig} 
\usepackage{cases} 
\usepackage{bm}
\usepackage{algorithm}
\usepackage{algorithmic}

\newtheorem{remark}{Remark}[section] 
\newtheorem{example}{Example}[section] 

\usepackage{booktabs}
\usepackage{mathtools}

\title{Modified BDF2 schemes for subdiffusion models with a singular source term}

\author{Minghua Chen \thanks{School of Mathematics and Statistics, Gansu Key Laboratory of Applied Mathematics and Complex Systems,
 Lanzhou University, Lanzhou 730000, P.R. China. 
Email address: chenmh@lzu.edu.cn}
 \and  Jiankang Shi \thanks{School of Mathematics and Statistics, Gansu Key Laboratory of Applied Mathematics and Complex Systems,
 Lanzhou University, Lanzhou 730000, P.R. China. 
Email address: shijk20@lzu.edu.cn}
 \and  Zhi Zhou\thanks{Department of Applied Mathematics, The Hong Kong Polytechnic University, Kowloon, Hong Kong, P.R. China. 
Email address: zhizhou@polyu.edu.hk}
}

\begin{document}

\maketitle

\begin{abstract}
The aim of this paper is to study the time stepping scheme for
approximately solving the subdiffusion equation with a weakly singular source term.
In this case, many popular time stepping schemes, including the correction of high-order  BDF methods,
may lose their high-order accuracy.
To fill in this gap, in this paper, we develop a novel time stepping scheme, where
the source term is regularized by using a $k$-fold integral-derivative
and  the equation is discretized by using  a modified BDF2 convolution quadrature.
We prove that the proposed time stepping scheme is second-order,
even if the source term is nonsmooth in time and incompatible with the initial data.
Numerical results are presented to support the theoretical results.
 \end{abstract}



\pagestyle{myheadings}
\thispagestyle{plain}
\markboth{M. CHEN, J. SHI,  AND Z. ZHOU}{IDK-BDF2 FOR SUBDIFFUSION}

\section{Introduction}\label{Se:intro}

For anomalous, non-Brownian diffusion, a mean squared displacement often follows the following power-law
$$\langle x^2(t) \rangle \simeq K_\alpha t^\alpha.$$
Prominent examples for subdiffusion  include
the classical charge carrier transport in amorphous semiconductors, tracer diffusion in subsurface aquifers,
 porous systems, dynamics of a bead in a polymeric network,
 or the motion of passive tracers in living biological cells \cite{Metzler:19,Metzler:00}.  Subdiffusion of this type is characterised by a long-tailed waiting time
 probability density function  $\psi (t) \simeq  t^{-1-\alpha}$, corresponding to the time-fractional diffusion equation with and without an external force field
   \cite[Eq. (88)]{Metzler:00}
\begin{gather*}
\partial_t u(x,t)-\partial^{1-\alpha}_t Au(x,t)=f(x,t),~~0<\alpha<1. \tag{$\spadesuit$}
\end{gather*}
Here   $f$ is a given source function,  and the operator $A=\Delta$  denotes Laplacian on a polyhedral domain $\Omega \subset \mathbb{R}^d$ ($d=1,2,3$) with a homogenous Dirichlet boundary condition.
The fractional derivative is taken in the  Riemann-Liouville sense, that is, $\partial^{1-\alpha}_t f=\partial_t J^\alpha f$ with
 the fractional integration operator
 $$J^\alpha f(t)=\frac{1}{\Gamma(\alpha)}\int_0^t(t-\tau)^{\alpha-1}f(\tau)d\tau=\frac{1}{\Gamma(\alpha)}t^{\alpha-1}*f(t),$$
and $*$ denotes the Laplace convolution: $(f*g)(t)=\int_0^t f(t-\tau)g(\tau)d\tau$.

Since the Riemann-Liouvile fractional derivative and the Caputo fractional derivative can be written in the form   \cite[p. 76]{Podlubny:99}
$$\partial^{\alpha}_t u(x,t)={^C\!{D}}^{\alpha}_{t} u(x,t) +\frac{1}{\Gamma(1-\alpha)}t^{-\alpha}u(x,0),$$
which implies that the equivalent form of $(\spadesuit)$ can be rewritten  as
\begin{gather*}
\partial_t u(x,t)-{^C\!{D}}^{1-\alpha}_t Au(x,t)=f(x,t)+\frac{Au(x,0)}{\Gamma(\alpha)}t^{-(1-\alpha)},~~0<\alpha<1 \tag{$\heartsuit$}
\end{gather*}
with the Caputo fractional derivative
\begin{equation*}\label{Cfd}
  {^C\!{D}}^{\alpha}_{t}u(t) = \frac{1}{\Gamma(1-\alpha)} \int^{t}_{0} {(t-s)^{-\alpha}u'(s)} ds,~~  0<t\leq T.
\end{equation*}

Applying the fractional integration operator $J^{1-\alpha}$ to both sides of $(\spadesuit)$, we obtain the
equivalent form of $(\spadesuit)$ as, see \cite[Eq. (1.6)]{McLean:19} or \cite[Eq. (2.3)]{SZS:22}, namely,
\begin{gather*}
{^C\!{D}}^{\alpha}_{t} u(x,t) \!-\! A u(x,t)=\frac{1}{\Gamma(1-\alpha)}t^{-\alpha}*f(x,t)-\frac{J^\alpha Au(x,t)|_{t=0}}{\Gamma(1-\alpha)}t^{-\alpha}, ~0<\alpha<1. \tag{$\clubsuit$}
\end{gather*}

As another example, the fractal mobile/immobile models for solute transport associate with power law decay PDF describing random waiting times in the immobile zone,
leads to the following models \cite[Eq. (15)]{SBMB:03}
\begin{gather*}
 \partial_t u(x,t)+{^C\!{D}}^{\alpha}_{t} u(x,t) - A u(x,t)=-\frac{1}{\Gamma(1-\alpha)}t^{-\alpha}u(x,0), ~0<\alpha<1. \tag{$\diamondsuit$}
\end{gather*}

Note that  the right hand side in aforementioned PDE models  ($\spadesuit$)-($\diamondsuit$)
might be nonsmooth in the time variable. In this paper, we consider the subdiffusion model with weakly singular source term:
\begin{equation} \label{fee}
\begin{aligned}
{^C\!{D}}^{\alpha}_{t} u(x,t) - A u(x,t)=g(x,t):=t^{\mu}\circ f(x,t)
\end{aligned}
\end{equation}
with the initial condition  $u(x,0)=u_0(x):=v$,  and the homogeneous Dirichlet boundary conditions.
The symbol $\circ$ can be either the convolution $*$  or the  product, and $\mu$ is a parameter such that
$$\mu >-1~~ \text{if} ~\circ~ \text{denotes convolution},~~ \text{and}~~ \mu \geq -\alpha~~ \text{if} ~\circ~ \text{denotes product}.$$
The well-posedness could be proved using the separation of variables and Mittag--Leffler functions, see e.g. \cite[Eq. (2.11)]{SY:11}.

Note that many existing time stepping schemes may lose their high-order accuracy when the source term is nonsmooth in the time variable.
As an example, it was reported in \cite[Section 4.1]{JLZ:2017} that the convolution quadrature generated by $k$ step BDF method (with initial correction)
converges with order $O(\tau^{1+\mu})$, provided that the source term behaves like $t^{\mu}$, $\mu>0$, see Lemma 3.2 in \cite{WZ:2020}, also see Table \ref{table:6.1}.
 The aim of this paper is to fill in this gap.


It is well-known that the smoothness of all the data of \eqref{fee} (e.g., $f=0$) do not imply the smoothness of the solution $u$ which has an initial layer  at $t\rightarrow 0^{+}$ (i.e., unbounded near $t=0$)
\cite{Podlubny:99,SY:11,SOG:2017}.
There are already two predominant  discretization techniques in time direction to restore the desired convergence rate for subdiffusion  under appropriate regularity source function.
The first type is that the nonuniform time meshes/graded meshes are employed to compensate/capture  the singularity of
the continuous solution near $t=0$ under  the appropriate regularity  source function and initial data, see \cite{CJB:2021,Kopteva:2021,Liao:2018,McLeanMustapha:2015,MustaphaAbdallahFurati:2014,Mustapha:20,SOG:2017}.
See also spectral method with specially designed basis functions \cite{Chen:2020, HouXu:2017,ZayernouriKarniadakis:2013}.
The second  type is that,  based on correction of high-order  BDF$k$  or   $L_k$ approximation, the desired  high-order convergence rates can be restored even for nonsmooth initial data.
For fractional ODEs, one idea is to use starting quadrature weights to correct the fractional integrals  \cite{Lubich:86} (or fractional substantial calculus  \cite{CD:15})
 $$J^\alpha g(t)=\frac{1}{\Gamma(\alpha)}\int_0^t(t-\tau)^{\alpha-1}g(\tau)d\tau~~{\rm with}~~g(t)=t^\mu f(t),~~ \mu>-1,$$
where the algorithms  rely  on expanding the solution into power series of $t$.
For fractional PDEs, a common practice is to split the source term into
$$
g(t)=g(0)+\sum^{k-1}_{l=1}\frac{t^{l}}{l!}\partial^{l}_{t}g(0)+\frac{t^{k-1}}{(k-1)!} \ast \partial^{k}_{t}g.
$$
Then approximating $g(0)$ by $\partial_{\tau}J^1g(0)$ may to a modified BDF2 scheme with correction in the first step \cite{Cuesta:06}.
The correction of  high-order  BDF$k$  or   $L_k$ convolution quadrature are well developed in   \cite{JLZ:2017,SC:2020,YKF:2018}  when  the source term sufficiently smooth in the time variable.
Performing  the integral on both sides for \eqref{fee}, e.g,  approximate $u(t)$ by $\partial_{\tau}J^1u(t)$,
 a  second-order  time-stepping schemes  are given in \cite{ZT:2022}, where the singular source function is
$g(x,t)=t^{\mu}f(x)$ with a spatially dependent function $f$.
How to deal with a more general source term, which might be nonsmooth in the time variable, is still unavailable in the literature.

 In this paper,
we develop a novel second-order time stepping scheme (IDk-BDF2) for solving the  subdiffusion \eqref{fee} with a weakly singular  source term,
where the   low regularly  source term is regularized by using a $k$-fold integral-derivative (IDk)
and  the equation is discretized by using  a modified BDF2 convolution quadrature.
We prove that the proposed time stepping scheme is second-order,
even if the source term is nonsmooth in time and incompatible with the initial data.
Numerical results are presented to support the theoretical results.


The paper is organized as follows. In Section \ref{Se:corre}, we introduce the development of the IDk-BDF2 scheme for model \eqref{fee}.
In Section 3 and 4, based on operational calculus,  the detailed convergence analysis of IDk-BDF2 are provided, respectively,  for general source function   $f(x,t)$ and certain form  $t^{\mu}f(x)  $.
Then the desired results with the  low regularity source term  $t^{\mu}\circ f(x,t)$  are obtained in Section 5.
To show the effectiveness of the presented schemes, the results of numerical experiments are reported in Section 6.
Finally, we conclude the paper with some remarks in the last section.

\section{IDk-BDF2 Method}\label{Se:corre}
In this section, we first provide   IDk-BDF2 method  for solving subdiffusion  \eqref{fee} if the source term $g(x,t)$ possess the  mild regularity.
Let $V(t)=u(t)-v$ with $V(0) = 0$. Then the model \eqref{fee}  can be  rewritten   as
\begin{equation}\label{rfee}
 \partial^{\alpha}_{t} V(t) - A V(t)= Av + g(t), \quad 0<t\leq T.
\end{equation}

From \cite{LST:1996} and \cite{Thomee:2006}, we know that the operator $A$ satisfies the following resolvent estimate
\begin{equation*}\label{resolvent estimate}
\left\| (z-A)^{-1} \right\| \leq c_{\phi} |z|^{-1} \quad \forall z\in \Sigma_{\phi}
\end{equation*}
for all $\phi\in (\pi/2, \pi)$, where $\Sigma_{\theta}:=\{ z\in \mathbb{C}\backslash \{0\}:|\arg z| < \theta \}$ is a sector of the complex plane $\mathbb{C}$.
Hence, $z^{\alpha}\in \Sigma_{\theta '}$ with $\theta '=\alpha\theta<\theta<\pi$ for all $z\in \Sigma_{\theta}$. Then, there exists  a positive constant $c$ such that
\begin{equation}\label{fractional resolvent estimate}
\left\| \left(z^{\alpha}-A\right)^{-1} \right\| \leq c |z|^{-\alpha} \quad \forall z\in \Sigma_{\theta}.
\end{equation}

\subsection{Discretization schemes}
Let $G(t)=J^1g(t)$ and  $ \mathcal{G}(t)=J^2g(t) $.
By first fundamental theorem of calculus,  we may  rewrite \eqref{rfee} as
\begin{equation}\label{rrfee}
{\rm ID1~ Method:}~~ \partial^{\alpha}_{t} V(t) - A V(t)=\partial_{t}( t Av + G(t)), \quad 0<t\leq T,
\end{equation}
\begin{equation}\label{rrrfee}
{\rm ID2~ Method:}~~ \partial^{\alpha}_{t} V(t) - A V(t)=\partial^{2}_{t} \left( \frac{t^{2}}{2} Av + \mathcal{G}(t) \right), \quad 0<t\leq T.
\end{equation}

Let $t_{n}=n \tau, n=0,1,\ldots,N$, be a uniform partition of the time interval $[0,T]$ with the step size $\tau=\frac{T}{N}$, and let $u^{n}$ denote the approximation of $u(t)$ and $g^{n}=g(t_{n})$.
The convolution quadrature generated by BDF2 approximates the Riemann-Liouville fractional derivative $\partial^{\alpha}_{t}\varphi(t_n)$ by
\begin{equation}\label{2.1}
  \partial^{\alpha}_{\tau}\varphi^{n}:=\frac{1}{\tau^{\alpha}}\sum^{n}_{j=0} \omega_{j}\varphi^{n-j}
\end{equation}
with $\varphi^{n}=\varphi(t_{n})$. Here  the weights $\omega_{j}$ are the coefficients in the series expansion
\begin{equation}\label{2.2}
  \delta^{\alpha}_{\tau}(\xi)=\frac{1}{\tau^{\alpha}}\sum^{\infty}_{j=0}\omega_{j}\xi^{j} \quad{\rm with} \quad \delta_{\tau}(\xi):=\frac{1}{\tau}\left(\frac{3}{2}-2\xi + \frac{1}{2} \xi^2\right).
\end{equation}
Then   IDk-BDF2 method for  \eqref{rrfee} and \eqref{rrrfee}   are, respectively,  designed by
\begin{equation}\label{2.3}
~{\rm ID1-BDF2~ Method:}~~~~   \partial^{\alpha}_{\tau} V^{n} - AV^{n}= \partial_{\tau}( t_{n} Av + G^{n}).
\end{equation}
\begin{equation}\label{4.1}
~~~~~{\rm ID2-BDF2~ Method:}~~~~  \partial^{\alpha}_{\tau} V^{n} - AV^{n}= \partial^{2}_{\tau}  \left( \frac{t^{2}_{n}}{2} Av + \mathcal{G}^{n} \right).
\end{equation}
\begin{remark}
In the time semidiscrete approximation \eqref{2.3} and \eqref{4.1}, we require  $v\in \mathcal{D}(A)$, i.e., the initial data $v$ is reasonably smooth.
However one can use the schemes  \eqref{2.3} and  \eqref{4.1} to prove the error estimates with the nonsmooth data $v\in L^2(\Omega)$, see Theorems \ref{Theorem5.8} and \ref{Theorem5.8888}.
Here, we mainly focus on the time semidiscrete approximation  \eqref{2.3} and  \eqref{4.1}, since the spatial discretization is well understood.
For example, we choose $v_h=R_hv$ if $v\in \mathcal{D}(A)$ and $v_h=P_hv$ if $v\in L^2(\Omega)$ following \cite{Thomee:2006,WYY:2020}.
\end{remark}

\subsection{Solution representation for \eqref{rrfee} and \eqref{rrrfee}}
Taking the Laplace transform in both sides of  \eqref{rrfee}, it leads to
\begin{equation*}
\widehat{V}(z)= (z^{\alpha}-A)^{-1}\left( z^{-1}Av   + z \widehat{G}(z) \right).
\end{equation*}
By the inverse Laplace transform, there exists \cite{JLZ:2017}
\begin{equation}\label{LT}
V(t) =  \frac{1}{2\pi i} \int_{\Gamma_{\theta, \kappa}} e^{zt}  (z^{\alpha}-A)^{-1}\left( z^{-1}Av   + z \widehat{G}(z) \right) dz
\end{equation}
with
\begin{equation}\label{Gamma}
\Gamma_{\theta, \kappa}=\{ z\in \mathbb{C}: |z|=\kappa, |\arg z|\leq \theta \} \cup \{ z\in \mathbb{C}: z=re^{\pm i\theta}, r\geq \kappa \}
\end{equation}
and  $\theta \in (\pi/2, \pi)$, $\kappa>0$.

Similarly, applying  the Laplace transform in both sides of  \eqref{rrrfee}, it yields
\begin{equation*}
\widehat{V}(z)= (z^{\alpha}-A)^{-1}\left( z^{-1}Av   + z^2 \widehat{\mathcal{G}}(z) \right).
\end{equation*}
By the inverse Laplace transform, we obtain   
\begin{equation}\label{LT2}
\begin{split}
V(t)
&=  \frac{1}{2\pi i} \int_{\Gamma_{\theta, \kappa}} e^{zt}  (z^{\alpha}-A)^{-1}\left( z^{-1}Av   + z^{2} \widehat{\mathcal{G}}(z) \right) dz.
\end{split}
\end{equation}

\subsection{Discrete solution representation  for \eqref{2.3} and \eqref{4.1}}
Given a sequence $(\kappa_n)_0^\infty$  and take
$\widetilde{\kappa}(\zeta)=\sum_{n=0}^{\infty}\kappa_n \zeta^n$
to be its generating power series.
\begin{lemma}\label{Lemma2.1}
Let $\delta_{\tau}$ be  given in  \eqref{2.2} and $\gamma_{1}(\xi)=\frac{\xi}{(1-\xi)^2}$,  $G(t)=J^1g(t)$.
Then the discrete solution of \eqref{2.3} is  represented by
\begin{equation*}
V^{n}=\frac{1}{2\pi i}\int_{\Gamma^{\tau}_{\theta,\kappa}} e^{zt_n} (\delta^{\alpha}_{\tau}(e^{-z\tau})-A)^{-1} \delta_{\tau}(e^{-z\tau}) \tau \left( \gamma_{1}(e^{-z\tau}) \tau A v + \widetilde{G}(e^{-z\tau}) \right) dz
\end{equation*}
with $\Gamma^{\tau}_{\theta, \kappa}=\{z\in \Gamma_{\theta, \kappa}: |\Im z|\leq \pi / \tau\}$.
\end{lemma}
\begin{proof}
Multiplying the \eqref{2.3} by $\xi^{n}$ and summing over $n$ with $V^0=0$, we obtain
\begin{equation*}
\begin{split}
&\sum^{\infty}_{n=1} \partial^{\alpha}_{\tau} V^{n} \xi^{n} - \sum^{\infty}_{n=1}  AV^{n}  \xi^{n} = \sum^{\infty}_{n=1} \partial_{\tau}( t_{n} Av + G^{n}) \xi^{n}.
\end{split}
\end{equation*}
From  \eqref{2.1} and \eqref{2.2}, we have
\begin{equation*}
\begin{split}
\sum^{\infty}_{n=1} \partial^{\alpha}_{\tau} V^{n} \xi^{n}
=& \sum^{\infty}_{n=1} \frac{1}{\tau^{\alpha}}\sum^{n}_{j=0} \omega_{j} V^{n-j} \xi^{n}
=  \sum^{\infty}_{n=0} \frac{1}{\tau^{\alpha}}\sum^{n}_{j=0} \omega_{j} V^{n-j} \xi^{n}
=  \sum^{\infty}_{j=0} \frac{1}{\tau^{\alpha}}\sum^{\infty}_{n=j} \omega_{j} V^{n-j} \xi^{n}\\
=& \sum^{\infty}_{j=0} \frac{1}{\tau^{\alpha}}\sum^{\infty}_{n=0} \omega_{j} V^{n} \xi^{n+j}
=  \frac{1}{\tau^{\alpha}} \sum^{\infty}_{j=0}  \omega_{j}  \xi^{j} \sum^{\infty}_{n=0}  V^{n} \xi^{n}
=  \delta^{\alpha}_{\tau} (\xi) \widetilde{V} (\xi).
\end{split}
\end{equation*}
Similarly, one has
\begin{equation*}
\sum^{\infty}_{n=1} \partial_{\tau} t_{n} Av  \xi^{n} = \delta_{\tau}(\xi) \gamma_{1}(\xi) \tau Av, \quad
\sum^{\infty}_{n=1} \partial_{\tau}  G^{n} \xi^{n} =  \delta_{\tau}(\xi) \widetilde{G} (\xi)
\end{equation*}
with $\gamma_{1}(\xi)=\frac{\xi}{(1-\xi)^2}$.
It leads to
\begin{equation}\label{ads2.17}
\widetilde{V}(\xi) =   \left( \delta^{\alpha}_{\tau}(\xi) -A\right)^{-1} \delta_{\tau}(\xi) \left(\gamma_{1}(\xi) \tau Av + \widetilde{G} (\xi) \right).
\end{equation}
According to  Cauchy's integral formula, and the change of variables $\xi=e^{-z\tau}$, and  Cauchy's theorem, one has \cite{JLZ:2017}
\begin{equation}\label{DLT}
V^{n}=\frac{\tau}{2\pi i}\int_{\Gamma^{\tau}_{\theta,\kappa}} e^{zt_n} \left( \delta^{\alpha}_{\tau}(e^{-z\tau}) -A\right)^{-1} \delta_{\tau}(e^{-z\tau}) \left( \gamma_{1}(e^{-z\tau}) \tau Av + \widetilde{G} (e^{-z\tau}) \right)dz
\end{equation}
with $\Gamma^{\tau}_{\theta, \kappa}=\{z\in \Gamma_{\theta, \kappa}: |\Im z|\leq \pi / \tau\}$.
The proof is completed.
\end{proof}

\begin{lemma}\label{addLemma2.1}
Let $\delta_{\tau}$ be  given in  \eqref{2.2} and $\gamma_{2}(\xi)= \frac{\xi+ \xi^{2}}{(1-\xi)^3}$,  $\mathcal{G}(t)=J^2g(t)$.
Then the discrete solution of \eqref{4.1} is  represented by
\begin{equation*}
V^{n}=\frac{\tau}{2\pi i}\int_{\Gamma^{\tau}_{\theta,\kappa}} e^{zt_n} \left( \delta^{\alpha}_{\tau}(e^{-z\tau}) -A\right)^{-1} \delta^{2}_{\tau}(e^{-z\tau}) \left( \frac{\gamma_{2}(e^{-z\tau})}{2} \tau^{2} Av + \widetilde{\mathcal{G}} (e^{-z\tau}) \right)dz
\end{equation*}
with $\Gamma^{\tau}_{\theta, \kappa}=\{z\in \Gamma_{\theta, \kappa}: |\Im z|\leq \pi / \tau\}$.
\end{lemma}
\begin{proof}
Multiplying the \eqref{4.1} by $\xi^{n}$ and summing over $n$  with $V^0=0$, we obtain
\begin{equation*}
\begin{split}
&\sum^{\infty}_{n=1} \partial^{\alpha}_{\tau} V^{n} \xi^{n} - \sum^{\infty}_{n=1}  AV^{n}  \xi^{n} = \sum^{\infty}_{n=1} \partial^{2}_{\tau} \left( \frac{t^{2}_{n}}{2} Av + \mathcal{G}^{n} \right) \xi^{n}.
\end{split}
\end{equation*}
The similar arguments can be performed as Lemma \ref{Lemma2.1}, it yields
\begin{equation*}
\begin{split}
& \sum^{\infty}_{n=1} \partial^{\alpha}_{\tau} V^{n} \xi^{n} = \delta^{\alpha}_{\tau} (\xi) \widetilde{V} (\xi), \quad
  \sum^{\infty}_{n=1} \partial^{2}_{\tau} t^{2}_{n} Av  \xi^{n} = \delta^{2}_{\tau}(\xi) \gamma_{2}(\xi) \tau^{2} Av, \\
& \sum^{\infty}_{n=1} \partial^{2}_{\tau}  \mathcal{G}^{n} \xi^{n} =  \delta^{2}_{\tau}(\xi) \widetilde{\mathcal{G}} (\xi),~~\gamma_{2}(\xi)= \frac{\xi+ \xi^{2}}{(1-\xi)^3},
\end{split}
\end{equation*}
and
\begin{equation}\label{4.4}
\widetilde{V}(\xi) =   \left( \delta^{\alpha}_{\tau}(\xi) -A\right)^{-1} \delta^{2}_{\tau}(\xi) \left( \frac{\gamma_{2}(\xi)}{2} \tau^{2} Av + \widetilde{\mathcal{G}} (\xi) \right).
\end{equation}
Using    Cauchy's integral formula, and the change of variables $\xi=e^{-z\tau}$, and  Cauchy's theorem, one has
\begin{equation}\label{DLT2}
V^{n}=\frac{\tau}{2\pi i}\int_{\Gamma^{\tau}_{\theta,\kappa}} e^{zt_{n}} \left( \delta^{\alpha}_{\tau}(e^{-z\tau}) -A\right)^{-1} \delta^{2}_{\tau}(e^{-z\tau}) \left( \frac{\gamma_{2}(\xi)}{2} \tau^{2} Av + \widetilde{\mathcal{G}} (e^{-z\tau}) \right)dz
\end{equation}
with $\Gamma^{\tau}_{\theta, \kappa}=\{z\in \Gamma_{\theta, \kappa}: |\Im z|\leq \pi / \tau\}$.
The proof is completed.
\end{proof}

\section{Convergence analysis: General source function $g(x,t)$}\label{Se:conver}
In this section, we provide the detailed convergence analysis of  ID1-BDF2 in \eqref{2.3}  approximation for the subdiffusion \eqref{rrfee}, and ID2-BDF2 can be similarly augmented.
\subsection{A few technical lemmas}
First, we give some lemmas that will be used.
\begin{lemma}\cite{JLZ:2017}\label{Lemma 2.3}
Let $\delta_{\tau}(\xi)$ be given in \eqref{2.2}. Then there  exist the positive constants $c_{1},c_{2}$, $c$ and
$\theta \in (\pi/2, \theta_{\varepsilon})$ with  $\theta_{\varepsilon} \in (\pi/2, \pi),~\forall \varepsilon>0$ such that
\begin{equation*}
\begin{split}
& c_{1}|z|\le |\delta_{\tau}(e^{-z\tau})| \le c_{2}|z|, \quad   |\delta_{\tau}(e^{-z\tau})-z|\le c \tau^{2}|z|^{3}, \\
& |\delta^{\alpha}_{\tau}(e^{-z\tau})-z^{\alpha}|\le c \tau^{2}|z|^{2+\alpha},~\delta_{\tau}(e^{-z\tau}) \in \Sigma_{\pi/2+\varepsilon} \quad {\forall} z\in \Gamma^{\tau}_{\theta,\kappa}.
\end{split}
\end{equation*}
\end{lemma}

\begin{lemma}\label{Lemma nn3.3} 
Let $\delta_{\tau}(\xi)$ be given in \eqref{2.2}  and $\gamma_{l}(\xi)=\sum^{\infty}_{n=1}n^{l} \xi^{n}=\left( \xi\frac{d}{d\xi} \right)^{l} \frac{1}{1-\xi} $ with $l=0, 1, 2$. Then there exist a  positive constants $c$ such that
\begin{equation*}
\left| \frac{\gamma_{l}(e^{-z\tau})}{l!}  \tau^{l+1} - z^{-l-1} \right| \leq c \tau^{l+1} , \quad \forall z\in \Gamma^{\tau}_{\theta,\kappa},
\end{equation*}
where $\theta\in (\pi/2,\pi)$ is  sufficiently close to $\pi/2$.
\end{lemma}

\begin{proof}
The arguments can be performed in \cite{SC:2020} for $l=1, 2$. For $l=0$,
using
\begin{equation*}
\frac{1}{1-e^{-z\tau}}  \tau - z^{-1} =  \frac{z-\left(1-e^{-z\tau}\right)\tau^{-1}}{\left(1-e^{-z\tau}\right)\tau^{-1} z },
\end{equation*}
and  Lemma \ref{Lemma 2.3}, it yields  $|1-e^{-z\tau}|\geq c_1 |z|\tau$ and
\begin{equation}\label{3.0001}
\begin{split}
\left| (1-e^{-z\tau})\tau^{-1}z \right| \geq  c|z|^{2}~~\forall z\in\Gamma^{\tau}_{\theta,\kappa}.
\end{split}
\end{equation}
Since
\begin{equation*}
\begin{split}
\left| z-\left(1-e^{-z\tau}\right)\tau^{-1} \right|
& =  \left| z- \left( 1- \sum^{\infty}_{j=0} \frac{(-z\tau)^{j}}{j!}  \right) \tau^{-1} \right|
  =  \left| z- \left( - \sum^{\infty}_{j=1} \frac{(-z\tau)^{j}}{j!}  \right) \tau^{-1} \right| \\
& =  \left| z-  z \sum^{\infty}_{j=0} \frac{(-z\tau)^{j}}{(j+1)!}    \right| 
  =  \left| \tau z^{2} \sum^{\infty}_{j=0} \frac{(-z\tau)^{j}}{(j+2)!}   \right|  \leq c \tau   \left| z \right|^{2}.
\end{split}
\end{equation*}
Thus  we have
\begin{equation*}
\left| \frac{1}{1-e^{-z\tau}}  \tau - z^{-1} \right| \leq c \tau.
\end{equation*}
The proof is completed.
\end{proof}

\begin{lemma}\label{Lemma 3.4}
Let $\delta_{\tau}(\xi)$ be given in \eqref{2.2} and $\gamma_{l}(\xi)=\sum^{\infty}_{n=1}n^{l} \xi^{n}=\left( \xi\frac{d}{d\xi} \right)^{l} \frac{1}{1-\xi} $ with $l=0, 1, 2$.
Then there exist a  positive constants $c$ such that
\begin{equation}\label{nn3.2}
\left|\delta_{\tau}(e^{-z\tau}) \frac{\gamma_{l}(e^{-z\tau})}{l!} \tau^{l+1} - z^{-l} \right| \leq c \tau^{l+1} \left| z \right| + c \tau^{2}|z|^{2-l}, \quad \forall z\in \Gamma^{\tau}_{\theta,\kappa}, 
\end{equation}
where $\theta\in (\pi/2,\pi)$ is  sufficiently close to $\pi/2$.
\end{lemma}
\begin{proof}
Let
\begin{equation*}
\delta_{\tau}(e^{-z\tau}) \frac{\gamma_{l}(e^{-z\tau})}{l!}  \tau^{l+1} - z^{-l} = J_{1} + J_{2}
\end{equation*}
with
\begin{equation*}
J_{1}= \delta_{\tau}(e^{-z\tau}) \frac{\gamma_{l}(e^{-z\tau})}{l!}  \tau^{l+1} - \delta_{\tau}(e^{-z\tau}) z^{-l-1} \quad {\rm and} \quad
J_{2}= \delta_{\tau}(e^{-z\tau}) z^{-l-1} - z^{-l}.
\end{equation*}
According to Lemma \ref{Lemma 2.3} and \ref{Lemma nn3.3},  we have
\begin{equation*}
\left|J_{1} \right| =  \left| \delta_{\tau}(e^{-z\tau}) \left(\frac{\gamma_{l}(e^{-z\tau})}{l!}  \tau^{l+1} -  z^{-l-1} \right) \right| \leq c_{2} \left| z \right|  c \tau^{l+1} \leq c \tau^{l+1}  \left| z \right|
\end{equation*}
and
\begin{equation*}
\left|J_{2} \right| = \left| \left( \delta_{\tau}(e^{-z\tau})-z \right) z^{-l-1}  \right| \leq  c \tau^{2}|z|^{2-l}.
\end{equation*}
By the triangle inequality,  the desired result is obtained.
\end{proof}

\begin{lemma}\label{Lemma 3.11}
Let $\delta^{\alpha}_{\tau}$  be given by \eqref{2.2}  and $\gamma_{l}(\xi)=\sum^{\infty}_{n=1}n^{l} \xi^{n}=\left( \xi\frac{d}{d\xi} \right)^{l} \frac{1}{1-\xi} $ with $l=0,1, 2$. Then there exist a  positive constants $c$ such that
\begin{equation*}
\begin{split}
&\left\| \left( \delta^{\alpha}_{\tau}(e^{-z\tau}) -A\right)^{-1} \delta_{\tau}(e^{-z\tau}) \frac{\gamma_{l}(e^{-z\tau})}{l!} \tau^{l+1} - (z^{\alpha}-A)^{-1} z^{-l} \right\| \\
 &\quad \leq c \tau^{l+1}  \left| z \right|^{1-\alpha} + c \tau^{2}|z|^{2-l-\alpha}.
\end{split}
\end{equation*}
\end{lemma}
\begin{proof}
Let
\begin{equation*}
\left( \delta^{\alpha}_{\tau}(e^{-z\tau}) -A\right)^{-1} \delta_{\tau}(e^{-z\tau})  \frac{\gamma_{l}(e^{-z\tau})}{l!}  \tau^{l+1}   - (z^{\alpha}-A)^{-1} z^{-l} = I + II
\end{equation*}
with
\begin{equation*}
\begin{split}
I  & = \left( \delta^{\alpha}_{\tau}(e^{-z\tau}) -A\right)^{-1} \left[ \delta_{\tau}(e^{-z\tau})  \frac{\gamma_{l}(e^{-z\tau})}{l!}  \tau^{l+1}   -  z^{-l} \right],  \\
II & = \left[ \left( \delta^{\alpha}_{\tau}(e^{-z\tau}) -A\right)^{-1} - (z^{\alpha}-A)^{-1} \right]z^{-l}.
\end{split}
\end{equation*}
The resolvent estimate \eqref{fractional resolvent estimate} and Lemma  \ref{Lemma 2.3} imply directly
\begin{equation}\label{discrete fractional resolvent estimate}
\| \left(\delta^{\alpha}_{\tau}(e^{-z\tau}) -A\right)^{-1} \| \leq c |z|^{-\alpha}.
\end{equation}
From \eqref{discrete fractional resolvent estimate} and Lemma \ref{Lemma 3.4}, we obtain
\begin{equation*}
\|I\|  \leq c \tau^{l+1}  \left| z \right|^{1-\alpha} + c \tau^{2}|z|^{2-l-\alpha}. 
\end{equation*}
Using Lemma \ref{Lemma 2.3}, \eqref{discrete fractional resolvent estimate} and the identity
\begin{equation}\label{identity1}
\begin{split}
  &\left( \delta^{\alpha}_{\tau}(e^{-z\tau}) -A\right)^{-1} - (z^{\alpha}-A)^{-1}\\
 &\quad =\left( z^{\alpha} - \delta^{\alpha}_{\tau}(e^{-z\tau}) \right) \left( \delta^{\alpha}_{\tau}(e^{-z\tau}) -A\right)^{-1} (z^{\alpha}-A)^{-1},
\end{split}
\end{equation}
we estimate $II$ as following
\begin{equation*}
\|II\|  \leq c \tau^{2} |z|^{2+\alpha}  c |z|^{-\alpha} c |z|^{-\alpha}  |z|^{-l} \leq c \tau^{2} |z|^{2-l-\alpha}.
\end{equation*}
By the triangle inequality,  the desired result is obtained.
\end{proof}

\begin{lemma}\label{addLemma 3.6}
Let $\delta^{\alpha}_{\tau}$  be given by \eqref{2.2}  and $\gamma_{1}(\xi)=\sum^{\infty}_{n=1}n \xi^{n}=\left( \xi\frac{d}{d\xi} \right) \frac{1}{1-\xi}= \frac{\xi}{(1-\xi)^2}$. Then there exist a  positive constants $c$ such that
\begin{equation*}\label{addnn3.4}
\begin{split}
\left\| \left( \delta^{\alpha}_{\tau}(e^{-z\tau}) -A\right)^{-1} \delta_{\tau}(e^{-z\tau}) \gamma_{1}(e^{-z\tau}) \tau^{2} A - (z^{\alpha}-A)^{-1} z^{-1}  A \right\| \leq c \tau^{2} \left| z \right|.
\end{split}
\end{equation*}
\end{lemma}
\begin{proof}
Using identical
$(z^{\alpha}-A)^{-1} z^{-l}  A = - z^{-1}   + (z^{\alpha}-A)^{-1} z^{\alpha} z^{-1}$ and
\begin{equation*}
\begin{split}
 \left( \delta^{\alpha}_{\tau}(e^{-z\tau}) -A\right)^{-1} \delta_{\tau}(e^{-z\tau})  A
 =  - \delta_{\tau}(e^{-z\tau})  + \left( \delta^{\alpha}_{\tau}(e^{-z\tau}) -A\right)^{-1} \delta^{\alpha}_{\tau}(e^{-z\tau}) \delta_{\tau}(e^{-z\tau})  A,
\end{split}
\end{equation*}
we get
\begin{equation*}
\left( \delta^{\alpha}_{\tau}(e^{-z\tau}) -A\right)^{-1} \delta_{\tau}(e^{-z\tau}) \gamma_{1}(e^{-z\tau})  \tau^{2} A - (z^{\alpha}-A)^{-1} z^{-1}  A = J_{1} + J_{2} + J_{3} + J_{4}
\end{equation*}
with
\begin{equation*}
\begin{split}
J_{1} =& \left( \delta^{\alpha}_{\tau}(e^{-z\tau}) -A\right)^{-1} \delta^{\alpha}_{\tau}(e^{-z\tau}) \left( \delta_{\tau}(e^{-z\tau})  \gamma_{1}(e^{-z\tau})  \tau^{l+1} - z^{-1} \right), \\
J_{2} =& \left( \delta^{\alpha}_{\tau}(e^{-z\tau}) -A\right)^{-1} \left( \delta^{\alpha}_{\tau}(e^{-z\tau}) -  z^{\alpha} \right) z^{-1}, \\
J_{3} =& \left( \left( \delta^{\alpha}_{\tau}(e^{-z\tau}) -A\right)^{-1} - (z^{\alpha}-A)^{-1} \right) z^{\alpha - 1},
\quad J_{4} = z^{-1} - \delta_{\tau}(e^{-z\tau})  \gamma_{1}(e^{-z\tau})  \tau^{2}.
\end{split}
\end{equation*}


According to \eqref{discrete fractional resolvent estimate} and Lemmas \ref{Lemma 2.3}, \ref{Lemma 3.4} with $l=1$, we  estimate $J_{1}$, $J_{2}$ and $J_{4}$ as following
\begin{equation*}
\begin{split}
&\left\| J_{1}  \right\| \leq c |z|^{-\alpha} |z|^{\alpha}  \tau^{2} \left| z \right|   \leq  c \tau^{2} \left| z \right|, \\
&\left\| J_{2}  \right\| \leq c |z|^{-\alpha} \tau^{2}|z|^{2+\alpha} |z|^{-1}  \leq   c \tau^{2}|z|, \quad
\left\| J_{4}  \right\| \leq  c \tau^{2} \left| z \right|.
\end{split}
\end{equation*}
From  Lemma \ref{Lemma 2.3}, \eqref{discrete fractional resolvent estimate} and the identity \eqref{identity1},
we estimate $J_{3}$ as following
\begin{equation*}
\left\| J_{3}  \right\| \leq c \tau^{2}|z|^{2+\alpha}  |z|^{-\alpha} |z|^{-\alpha} |z|^{\alpha-1}  \leq   c \tau^{2}|z|.
\end{equation*}
By the triangle inequality,  the desired result is obtained.
\end{proof}

\subsection{Error analysis for general source function $g(x,t)$}
From   $G(t)=J^1g(t)$, the Taylor expansion of source  function with the remainder term in integral form:
\begin{equation*}
\begin{split}
1 \ast g(t)=G(t)&=  G(0)+t G'(0)+  \frac{t^2}{2} G''(0) + \frac{t^2}{2}  \ast  G'''(t)\\
&= J^1g(0)+t g(0)+  \frac{t^2}{2} g'(0) + \frac{t^2}{2}  \ast  g''(t).
\end{split}
\end{equation*}
Then  we obtain   the following results with $g^{(-1)}(0)=J^1g(0)$.
\begin{lemma}\label{lemma3.9}
Let $V(t_{n})$ and $V^{n}$ be the solutions of \eqref{rrfee} and \eqref{2.3}, respectively.
Let $v=0$ and $G(t):=\frac{t^{l}}{l!} g^{(l-1)}(0)$ with $l=0,1, 2$.  Then
\begin{equation}\label{nn3.7}
\left\| V(t_{n}) - V^{n} \right\| \leq \left( c\tau^{l+1}t_{n}^{\alpha-2} + c\tau^{2}t_{n}^{\alpha+l-3} \right) \left\| g^{(l-1)}(0)\right\|.
\end{equation}
\end{lemma}
\begin{proof}
Using    \eqref{LT} and \eqref{DLT}, there exist
\begin{equation*}
V(t_{n})=\frac{1}{2\pi i}\int_{\Gamma_{\theta,\kappa}}{e^{zt_{n}}(z^{\alpha}-A)^{-1} \frac{1}{z^{l}} g^{(l-1)}(0)}dz,
\end{equation*}
and
\begin{equation*}
V^{n}=\frac{1}{2\pi i}\int_{\Gamma^{\tau}_{\theta,\kappa}} e^{zt_{n}} \left( \delta^{\alpha}_{\tau}(e^{-z\tau}) -A\right)^{-1} \delta_{\tau}(e^{-z\tau}) \frac{\gamma_{l}(e^{-z\tau})}{l!} \tau^{l+1} g^{(l-1)}(0) dz,
\end{equation*}
where  $\theta\in (\pi/2,\pi)$ is  sufficiently close to $\pi/2$,
and $\gamma_{l}(\xi)=\sum^{\infty}_{n=1}n^{l} \xi^{n}$.
Let
\begin{equation*}
V(t_{n})-V^{n}=J_1 + J_2
\end{equation*}
with
\begin{equation*}
\begin{split}
J_1
\!=\!\frac{1}{2\pi i}\int_{\Gamma^{\tau}_{\theta,\kappa}}\!\!\!e^{zt_{n}}\!\left[\!\frac{\left(z^{\alpha} - A\right)^{-1} }{z^{l}}
\!- \!\left( \delta^{\alpha}_{\tau}(e^{-z\tau}) \!-\!A\right)^{-1} \delta_{\tau}(e^{-z\tau}) \frac{\gamma_{l}(e^{-z\tau})}{l!} \tau^{l+1} \right] \! g^{(l-1)}(0) dz,
\end{split}
\end{equation*}
and
\begin{equation*}
\begin{split}
J_2
=\frac{1}{2\pi i}\int_{\Gamma_{\theta,\kappa}\setminus\Gamma^{\tau}_{\theta,\kappa}}{e^{zt_{n}}\left(z^{\alpha} - A\right)^{-1} \frac{1}{z^{l}} g^{(l-1)}(0)}dz.
\end{split}
\end{equation*}
According to the triangle inequality, \eqref{fractional resolvent estimate} and Lemma \ref{Lemma 3.11}, one has
\begin{equation*}
\begin{split}
  \| J_1 \|
  &  \leq c  \int^{\frac{\pi}{\tau\sin\theta}}_{\kappa} e^{rt_{n}\cos\theta} \left(\tau^{l+1}  r^{1-\alpha} +  \tau^{2} r^{2-l-\alpha} \right) dr \left\|g^{(l-1)}(0)\right\| \\
  & \quad + c \int^{\theta}_{-\theta}e^{\kappa t_{n} \cos\psi} \left(\tau^{l+1}  \kappa^{2-\alpha} +  \tau^{2} \kappa^{3-l-\alpha} \right) d\psi  \left\|g^{(l-1)}(0)\right\| \\
  &  \leq  \left( c\tau^{l+1}t_{n}^{\alpha-2} + c\tau^{2}t_{n}^{\alpha+l-3} \right) \left\| g^{(l-1)}(0)\right\|,
\end{split}
\end{equation*}
for the last inequality,   we use
\begin{equation}\label{ad3.3.09}
\begin{split}
& \int^{\frac{\pi}{\tau\sin\theta}}_{\kappa} e^{rt_{n}\cos\theta} r^{2-l-\alpha}dr = t_n^{\alpha+l-3} \int^{\frac{t_n\pi}{\tau\sin\theta}}_{t_n\kappa} e^{s\cos\theta} s^{2-l- \alpha}ds  \leq c t_n^{\alpha+l-3},\\
& \int^{\theta}_{-\theta}e^{\kappa t_{n} \cos\psi} \kappa^{3-l- \alpha} d\psi= t_n^{\alpha+l-3} \int^{\theta}_{-\theta}e^{\kappa t_{n} \cos\psi} \left(\kappa t_{n}\right)^{3-l- \alpha} d\psi \leq c t_n^{\alpha+l-3}.
\end{split}
\end{equation}
From   \eqref{fractional resolvent estimate}, it yields
\begin{equation*}
\begin{split}
\|J_2 \|
&\leq c \left\| g^{(l-1)}(0) \right\| \int^{\infty}_{\frac{\pi}{\tau\sin\theta}} e^{rt_{n}\cos\theta}r^{-l-\alpha}dr\\
&\leq c\tau^{2} \left\| g^{(l-1)}(0) \right\|  \int^{\infty}_{\frac{\pi}{\tau\sin\theta}} e^{rt_{n}\cos\theta}r^{2-l-\alpha}dr
\leq  c\tau^{2}t_{n}^{\alpha+l-3} \left\| g^{(l-1)}(0) \right\|.
\end{split}
\end{equation*}
Here we using $1\leq (\frac{\sin \theta}{\pi})^{2} \tau^{2} r^{2}$ with $r\geq \frac{\pi}{\tau \sin \theta}$.
The proof is completed.
\end{proof}

\begin{lemma}\label{lemma3.10}
Let $V(t_{n})$ and $V^{n}$ be the solutions of \eqref{rrfee} and \eqref{2.3}, respectively.
Let $v=0$,  $G(t):=\frac{t^{2}}{2} \ast  g''(t)$ and  $\int_{0}^{t}  (t-s)^{\alpha-1} \| g''(s) \| ds <\infty$. Then 
\begin{equation*}
\left\|V(t_{n})-V^{n}\right\|\leq c\tau^{2} \int_{0}^{t_{n}} (t_n-s)^{\alpha-1} \left\| g''(s) \right\|ds.
\end{equation*}
\end{lemma}
\begin{proof}
By  \eqref{LT}, we obtain
\begin{equation}\label{nas3.6}
\begin{split}
V(t_{n})
&=\frac{1}{2\pi i}\int_{\Gamma_{\theta,\kappa}}{e^{zt_{n}}(z^{\alpha} - A)^{-1} z \widehat{G}(z)}dz=(\mathscr{E}(t)\ast G(t))(t_{n})\\
&=\left(\mathscr{E}(t)\ast \left(\frac{t^{2}}{2} \ast  g''(t) \right)\right)(t_{n}) =\left(\left(\mathscr{E}(t)\ast \frac{t^{2}}{2}   \right)\ast  g''(t) \right)(t_{n})
\end{split}
\end{equation}
with
\begin{equation}\label{nas3.007}
  \mathscr{E}(t)= \frac{1}{2\pi i} \int_{\Gamma_{\theta,\kappa}} e^{zt}(z^{\alpha} - A)^{-1} z dz.
\end{equation}
From  \eqref{ads2.17}, it yields
\begin{equation*}
\begin{split}
\widetilde{V}(\xi)
&=\left( \delta^{\alpha}_{\tau}(\xi) -A\right)^{-1} \delta_{\tau}(\xi) \widetilde{G}(\xi) = \widetilde{\mathscr{E_{\tau}}}(\xi)\widetilde{G}(\xi)
=\sum^{\infty}_{n=0}\mathscr{E}^{n}_{\tau}\xi^{n}\sum^{\infty}_{j=0}G^j\xi^{j}\\
&=\sum^{\infty}_{n=0}\sum^{\infty}_{j=0}\mathscr{E}^{n}_{\tau} G^j \xi^{n+j}=\sum^{\infty}_{j=0}\sum^{\infty}_{n=j}\mathscr{E}^{n-j}_{\tau} G^j \xi^{n}
=\sum^{\infty}_{n=0}\sum^{n}_{j=0}\mathscr{E}^{n-j}_{\tau} G^j \xi^{n}=\sum^{\infty}_{n=0}V^n\xi^{n}
\end{split}
\end{equation*}
with
\begin{equation*}
V^{n}=\sum^{n}_{j=0}\mathscr{E}^{n-j}_{\tau} G^j:=\sum^{n}_{j=0}\mathscr{E}^{n-j}_{\tau} G(t_{j}).
\end{equation*}
Here $\sum^{\infty}_{n=0}\mathscr{E}^{n}_{\tau}\xi^{n}=\widetilde{\mathscr{E_{\tau}}}(\xi)=\left( \delta^{\alpha}_{\tau}(\xi) -A\right)^{-1} \delta_{\tau}(\xi)$.
From  the Cauchy's integral formula and the change of variables $\xi=e^{-z\tau}$, we obtain the representation of the $\mathscr{E}^{n}_{\tau}$ as following 
 \begin{equation*}
\mathscr{E}^{n}_{\tau}=\frac{1}{2\pi i}\int_{|\xi|=\rho}{\xi^{-n-1}\widetilde{\mathscr{E_{\tau}}}(\xi)}d\xi
=\frac{\tau}{2\pi i}\int_{\Gamma^{\tau}_{\theta,\kappa}} {e^{zt_n}\left( \delta^{\alpha}_{\tau}(e^{-z\tau}) -A\right)^{-1} \delta_{\tau}(e^{-z\tau}) }dz,
\end{equation*}
where  $\theta\in (\pi/2,\pi)$ is  sufficiently close to $\pi/2$ and $\kappa=t_{n}^{-1}$ in \eqref{Gamma}.

According to \eqref{discrete fractional resolvent estimate}, Lemma \ref{Lemma 2.3} and $\tau t^{-1}_{n} = \frac{1}{n}\leq 1$, there exists
\begin{equation}\label{3.0002}
\|\mathscr{E}^{n}_{\tau}\| \leq c \tau \left( \int^{\frac{\pi}{\tau\sin\theta}}_{\kappa} e^{rt_{n}\cos\theta} r^{1-\alpha}dr +\int^{\theta}_{-\theta}e^{\kappa t_{n}\cos\psi} \kappa^{2-\alpha}  d\psi\right)
\leq c\tau t_{n}^{\alpha-2} \leq c t_{n}^{\alpha-1}.
\end{equation}
Let $ \mathscr{E}_{\tau}(t)=\sum^{\infty}_{n=0}\mathscr{E}^{n}_{\tau}\delta_{t_{n}}(t)$, with $\delta_{t_{n}}$ being the Dirac delta function at $t_{n}$.
Then
\begin{equation}\label{nas3.8}
\begin{split}
(\mathscr{E}_{\tau}(t)\ast G(t))(t_{n})
& = \left(\sum^{\infty}_{j=0}\mathscr{E}^{j}_{\tau}\delta_{t_{j}}(t) \ast G(t) \right)(t_{n})\\
& = \sum^{n}_{j=0}\mathscr{E}^{j}_{\tau} G(t_{n}-t_{j})
  = \sum^{n}_{j=0}\mathscr{E}^{n-j}_{\tau} G(t_{j})=V^{n}.
\end{split}
\end{equation}
Moreover, using the above equation, there exist
\begin{equation*}
\begin{split}
  \widetilde{(\mathscr{E}_{\tau}\ast t^{l})}(\xi)
& = \sum^{\infty}_{n=0} \sum^{n}_{j=0}\mathscr{E}^{n-j}_{\tau}t^{l}_{j}\xi^{n}  =\sum^{\infty}_{j=0} \sum^{\infty}_{n=j}\mathscr{E}^{n-j}_{\tau}t^{l}_{j}\xi^{n}
  =\sum^{\infty}_{j=0} \sum^{\infty}_{n=0}\mathscr{E}^{n}_{\tau}t^{l}_{j}\xi^{n+j}\\
& =\sum^{\infty}_{n=0}\mathscr{E}^{n}_{\tau}\xi^{n}\sum^{\infty}_{j=0}t^{l}_{j}\xi^{j}  =\widetilde{\mathscr{E_{\tau}}}(\xi) \tau^{l} \sum^{\infty}_{j=0}j^{l}\xi^{j}
  =\widetilde{\mathscr{E}_{\tau}}(\xi) \tau^{l} \gamma_{l}(\xi).
\end{split}
\end{equation*}
From   \eqref{nas3.6}, \eqref{nas3.8} and  \eqref{nn3.7}, we have the following estimate
\begin{equation}\label{nad3.10}
\left\|\left((\mathscr{E}_{\tau}-\mathscr{E}) \ast \frac{t^{l}}{l!} \right)(t_n)\right\| \leq c\tau^{l+1}t_{n}^{\alpha-2} + c\tau^{2}t_{n}^{\alpha+l-3} \leq c\tau^{l}t_{n}^{\alpha-1} \quad l=0,1, 2.
\end{equation}

Next, we prove the following inequality \eqref{3.0003}  for $t>0$
\begin{equation}\label{3.0003}
\left\|\left((\mathscr{E}_{\tau}-\mathscr{E}) \ast \frac{t^{2}}{2} \right)(t)\right\| \leq c\tau^{2}t^{\alpha-1},\quad \forall t\in (t_{n-1},t_{n}).
\end{equation}
By Taylor series expansion of $\mathscr{E}(t)$ at $t=t_{n}$, we get
\begin{equation*}
\begin{split}
 \left( \mathscr{E} \ast \frac{t^{2}}{2} \right)(t)
=&\left( \mathscr{E} \ast \frac{t^{2}}{2} \right)(t_{n})+ (t-t_{n}) \left( \mathscr{E} \ast t \right)(t_{n}) \\
& + \frac{(t-t_{n})^{2}}{2} \left( \mathscr{E} \ast 1 \right)(t_{n})
 + \frac{1}{2}\int^{t}_{t_{n}}(t-s)^{2} \mathscr{E}(s)ds,
\end{split}
\end{equation*}
which  also holds  for $  \left( \mathscr{E}_{\tau} \ast t^{2} \right)(t) $.
Therefore, using  \eqref{nad3.10},  it yields
\begin{equation*}
\left\|\left((\mathscr{E}_{\tau}-\mathscr{E}) \ast \frac{t^{l}}{l!} \right)(t_{n})\right\| \leq c\tau^{l+1}t_{n}^{\alpha-2} + c\tau^{2}t_{n}^{\alpha+l-3} \leq c\tau^{l}t_{n}^{\alpha-1} \leq c\tau^{l}t^{\alpha-1}.
\end{equation*}
According to  \eqref{nas3.007}, \eqref{fractional resolvent estimate} and \eqref{ad3.3.09}, one has
\begin{equation*}
\begin{split}
\| \mathscr{E}(t) \|
\leq c \left( \int^{\infty}_{\kappa}e^{rt\cos\theta}r^{1-\alpha}dr + \int^{\theta}_{-\theta}e^{\kappa t \cos\psi}\kappa^{2-\alpha}d\psi \right)
\leq c t^{\alpha-2}.
\end{split}
\end{equation*}
Moreover, we get
\begin{equation*}
\left\| \int^{t}_{t_{n}}(t-s)^{2} \mathscr{E}(s)ds \right\| \leq c \int^{t_{n}}_{t}(s-t)^{2} s^{\alpha-2}ds \leq c \int^{t_{n}}_{t}(s-t) s^{\alpha-1}ds \leq  c \tau^{2} t^{\alpha-1}.
\end{equation*}
Using  the definition of $ \mathscr{E}_{\tau}(t)=\sum^{\infty}_{n=0}\mathscr{E}^{n}_{\tau}\delta_{t_{n}}(t)$ in \eqref{nas3.8} and \eqref{3.0002}, we deduce
\begin{equation*}
\left\| \int^{t}_{t_{n}}(t-s)^{2} \mathscr{E}_{\tau}(s)ds \right\|\leq  (t_n-t)^{2}  \| \mathscr{E}^{n}_{\tau} \|  \leq c \tau^{3} t_{n}^{\alpha-2} \leq c \tau^{2} t_{n}^{\alpha-1} \leq c \tau^{2} t^{\alpha-1}, ~ \forall \ t\!\in \! (t_{n-1},t_{n}).
\end{equation*}
By \eqref{nad3.10} and the above inequalities, it yields the inequality \eqref{3.0003}.
The proof is completed.
\end{proof}

\begin{theorem}[ID1-BDF2]\label{addtheorema3.1}
Let $V(t_{n})$ and $V^{n}$ be the solutions of \eqref{rrfee} and \eqref{2.3}, respectively. Let $v\in L^{2}(\Omega)$, $g\in C^{1}([0,T]; L^{2}(\Omega))$ and $\int_{0}^{t}  (t-s)^{\alpha-1} \left\| g''(s) \right\| ds <\infty$.  Then the following error estimate holds for any $t_n>0$:
\begin{equation*}
\begin{split}
&\left\|V^{n}-V(t_{n})\right\|\\
&\quad\leq c\tau^{2}   \left( t^{-2}_{n} \|v\|  +  t^{\alpha-2}_{n} \left\|  g(0) \right\|   +  t^{\alpha-1}_{n} \left\|  g'(0) \right\|  +  \int_{0}^{t_{n}} (t_n-s)^{\alpha-1} \left\| g''(s) \right\| ds \right).
\end{split}
\end{equation*}
\end{theorem}
\begin{proof}
Subtracting \eqref{LT} from \eqref{DLT}, we obtain
\begin{equation*}
V^{n}-V(t_{n})=I_{1}-I_{2}+I_{3}
\end{equation*}
with
\begin{equation*}
\begin{split}
I_{1} =& \frac{1}{2\pi i} \int_{\Gamma^{\tau}_{\theta,\kappa}} e^{zt_{n}}  \left[ \left( \delta^{\alpha}_{\tau}(e^{-z\tau}) -A\right)^{-1} \delta_{\tau}(e^{-z\tau})   \gamma_{1}(e^{-z\tau}) \tau^{2}  - (z^{\alpha}-A)^{-1} z^{-1}\right] Av dz,\\
I_{2} =& \frac{1}{2\pi i} \int_{\Gamma_{\theta,\kappa}\backslash \Gamma^{\tau}_{\theta,\kappa}} e^{zt_{n}}  (z^{\alpha}-A)^{-1} z^{-1}  Av dz,\\
I_{3} =& \frac{\tau}{2\pi i}\int_{\Gamma^{\tau}_{\theta,\kappa}} e^{zt_n} \left( \delta^{\alpha}_{\tau}(e^{-z\tau}) -A\right)^{-1} \delta_{\tau}(e^{-z\tau})  \widetilde{G} (e^{-z\tau}) dz \\
& - \frac{1}{2\pi i} \int_{\Gamma_{\theta, \kappa}} e^{zt_n}  (z^{\alpha}-A)^{-1} z \widehat{G}(z) dz.
\end{split}
\end{equation*}
According to the  Lemma \ref{addLemma 3.6}, we estimate the first term $I_{1}$ as following
\begin{equation}\label{add3.161}
\begin{split}
\left\|I_{1}\right\|
\leq & c\tau^{2} \left\| v \right\|  \int_{\Gamma^{\tau}_{\theta,\kappa}} \left|e^{zt_{n}}\right| |z| |dz| \\
\leq & c\tau^{2} \left\| v \right\|  \left( \int^{\frac{\pi}{\tau\sin\theta}}_{\kappa} e^{r t_{n}\cos\theta} r dr
    +  \int^{\theta}_{-\theta} e^{\kappa t_{n}\cos\psi} \kappa^{2} d\psi \right)\\
\leq &  c\tau^{2} t_{n}^{-2} \left\| v \right\|.
\end{split}
\end{equation}
Using the resolvent estimate \eqref{fractional resolvent estimate}, we estimate the second term $I_{2}$ as following
\begin{equation} \label{add3.162}
\left\|I_{2}\right\|
\leq c\int_{\Gamma_{\theta,\kappa}\backslash \Gamma^{\tau}_{\theta,\kappa}} {\left|e^{zt_{n}}\right||z|^{-1} \left\| v \right\|_{L^2(\Omega)} } |dz|
\leq c\tau^{2} t^{-2}_{n}\left\| v \right\|_{L^2(\Omega)},
\end{equation}
since
\begin{equation}\label{add3.16}
\begin{split}
\int_{\Gamma_{\theta,\kappa}\backslash \Gamma^{\tau}_{\theta,\kappa}} \left|e^{zt_{n}}\right||z|^{-1}  |dz|
& = \int^{\infty}_{\frac{\pi}{\tau \sin \theta}} e^{r t_{n} \cos \theta} r ^{-1}  dr \\
& \leq  c \tau^{2} \int^{\infty}_{\frac{\pi}{\tau \sin \theta}} e^{r t_{n} \cos \theta} r   dr \leq c \tau^{2} t^{-2}_{n}
\end{split}
\end{equation}
with $1\leq (\frac{\sin \theta}{\pi})^{2} \tau^{2} r^{2}$, $r\tau\geq \frac{\pi}{\sin \theta}$.

From Lemmas \ref{lemma3.9} and \ref{lemma3.10} with $G(t) = t g(0)+  \frac{t^2}{2} g'(0) + \frac{t^2}{2}  \ast  g''(t)$, there exist
\begin{equation*}
\left\|I_{3}\right\|\leq c\tau^{2}t_{n}^{\alpha-2} \left\|  g(0) \right\| + c\tau^{2}t_{n}^{\alpha-1} \left\|  g'(0) \right\| + c\tau^{2}  \int_{0}^{t_{n}}  (t_n-s)^{\alpha-1} \left\| g''(s) \right\|  ds.
\end{equation*}
The proof is completed.
\end{proof}

\begin{theorem}[ID2-BDF2]\label{addtheorema3.2}
Let $V(t_{n})$ and $V^{n}$ be the solutions of \eqref{rrrfee} and \eqref{4.1}, respectively. Let $v\in L^{2}(\Omega)$, $g\in C^{1}([0,T]; L^{2}(\Omega))$ and $\int_{0}^{t}  (t-s)^{\alpha-1} \left\| g''(s) \right\| ds <\infty$.  Then the following error estimate holds for any $t_n>0$:
\begin{equation*}
\begin{split}
&\left\|V^{n}-V(t_{n})\right\|\\
&\quad\leq c\tau^{2}   \left( t^{-2}_{n} \|v\|  +  t^{\alpha-2}_{n} \left\|  g(0) \right\|   +  t^{\alpha-1}_{n} \left\|  g'(0) \right\|  +  \int_{0}^{t_{n}} (t_n-s)^{\alpha-1} \left\| g''(s) \right\| ds \right).
\end{split}
\end{equation*}
\end{theorem}
\begin{proof}
Similar arguments can be performed as  Theorem \ref{addtheorema3.1}, we omit it here.
\end{proof}

\section{Convergence analysis: Singular  source function $t^{\mu}q(x)$, $\mu\geq -\alpha$}\label{Se:WSST}
Form Theorem \ref{addtheorema3.1} and Theorem \ref{addtheorema3.2}, it seems that there are no difference between  ID1-BDF2 and ID2-BDF2
for general source function. However, both of them are very different for the  singular  source function with the form  $t^{\mu}q(x)$.

\subsection{Low regularity source term}
In the section, we first consider low regularity source term $g(x,t)=t^{\mu}q(x)$ with $\mu >0$  for subdiffusion \eqref{rrfee}.
We  introduce the polylogarithm function or Bose-Einstein integral
\begin{equation}\label{polylogarithm function}
Li_{p}(\xi)= \sum_{j=1}^{\infty} \frac{\xi^{j}}{j^{p}},~~p\notin \mathbb{N}.
\end{equation}
\begin{lemma}\cite{JLZ:2016,YKF:2018}\label{addLemma:LipCA}
Let $|z\tau|\leq \frac{\pi}{\sin\theta}$ and  $\theta>\pi/2$ be close to $\pi/2$, and $p\neq 1, 2, \ldots$. The series
\begin{equation}\label{LpSE}
Li_{p}(e^{-z\tau}) = \Gamma(1-p)(z\tau)^{p-1} + \sum_{j=0}^{\infty} (-1)^{j} \zeta(p-j)\frac{(z\tau)^{j}}{j!}
\end{equation}
converges absolutely. Here $\zeta$ denotes the Riemann zeta function, namely, $\zeta(p)=Li_{p}(1)$.
\end{lemma}

Let $G(t)=J^1g(t) = \frac{t^{\mu+1}}{\mu+1}q$. Using  $\widehat{G}(z) =  \frac{\Gamma(\mu+1)}{z^{\mu+2}} q$
and  \eqref{LT}, we have 
\begin{equation}\label{add4.1}
\begin{split}
V(t)
&= \frac{1}{2\pi i} \int_{\Gamma_{\theta, \kappa}} e^{zt}  (z^{\alpha}-A)^{-1}\left( z^{-1}Av   + \frac{\Gamma(\mu+1)}{z^{\mu+1}} q \right) dz.
\end{split}
\end{equation}

From  \eqref{DLT}, the  discrete solution for the  subdiffusion \eqref{2.3} is
\begin{equation}\label{add4.4}
V^{n}=\frac{1}{2\pi i}\int_{\Gamma^{\tau}_{\theta,\kappa}} e^{zt_{n}} (\delta^{\alpha}_{\tau}(e^{-z\tau})-A)^{-1} \delta_{\tau}(e^{-z\tau}) \tau \left( \gamma_{1}(e^{-z\tau}) \tau A v + \widetilde{G}(e^{-z\tau}) \right) dz
\end{equation}
with $\gamma_{1}(e^{-z\tau})=\frac{e^{-z\tau}}{\left(1-e^{-z\tau}\right)^2}$ and $\Gamma^{\tau}_{\theta, \kappa}=\{z\in \Gamma_{\theta, \kappa}: |\Im z|\leq \pi / \tau\}$.
Here
\begin{equation*}
\widetilde{G}(\xi)
= \sum^{\infty}_{n=1} G^{n} \xi^{n}
= q \frac{\tau^{\mu+1}}{\mu+1}  \sum^{\infty}_{n=1} \frac{\xi^{n}}{n^{-\mu-1}} = q \frac{\tau^{\mu+1}}{\mu+1} Li_{-\mu-1}(\xi) \quad  {\rm with} \quad 0< \mu <1.
\end{equation*}

\begin{lemma}\label{addLemma 4.3}
Let $\delta^{\alpha}_{\tau}$  is given by \eqref{2.2} and $\gamma_{l}(\xi)=\sum^{\infty}_{n=1}n^{l} \xi^{n}=\left( \xi\frac{d}{d\xi} \right)^{l} \frac{1}{1-\xi}$ with $l=1, 2$ are given by Lemma \ref{Lemma 3.4}. Then there exist a  positive constants $c$ such that
\begin{equation*}
\begin{split}
&  \left\| \left( \delta^{\alpha}_{\tau}(e^{-z\tau}) -A\right)^{-1} \delta^{l}_{\tau}(e^{-z\tau})    - (z^{\alpha}-A)^{-1} z^{l} \right\| \leq c\tau^{2}|z|^{l+2-\alpha},\\
&  \left\|\left( \delta^{\alpha}_{\tau}(e^{-z\tau}) -A\right)^{-1} \delta^{l}_{\tau}(e^{-z\tau}) \frac{\gamma_{l}(e^{-z\tau})}{l!}  \tau^{l+1}  - (z^{\alpha}-A)^{-1} z^{-1} \right\| \leq c\tau^{2}|z|^{1-\alpha}
\quad \forall z\in \Gamma^{\tau}_{\theta,\kappa},
\end{split}
\end{equation*}
where $\theta\in (\pi/2,\pi)$ is  sufficiently close to $\pi/2$.
\end{lemma}
\begin{proof}
First we consider
\begin{equation*}
 \left( \delta^{\alpha}_{\tau}(e^{-z\tau}) -A\right)^{-1} \delta^{l}_{\tau}(e^{-z\tau})    - (z^{\alpha}-A)^{-1} z^{l} = I + II
\end{equation*}
with
\begin{equation*}
\begin{split}
I  & = \left( \delta^{\alpha}_{\tau}(e^{-z\tau}) -A\right)^{-1} \left( \delta^{l}_{\tau}(e^{-z\tau})   -  z^{l} \right),  \\
II & = \left( \left( \delta^{\alpha}_{\tau}(e^{-z\tau}) -A\right)^{-1} - (z^{\alpha}-A)^{-1} \right) z^{l}.
\end{split}
\end{equation*}
According to \eqref{discrete fractional resolvent estimate} and Lemma \ref{Lemma 2.3}, we obtain
\begin{equation*}
\|I\| \leq  c \tau^{2} |z|^{l+2-\alpha}.
\end{equation*}
Using the Lemma \ref{Lemma 2.3}, \eqref{discrete fractional resolvent estimate}, \eqref{fractional resolvent estimate} and the identity
$$\left( \delta^{\alpha}_{\tau}(e^{-z\tau}) -A\right)^{-1} - (z^{\alpha}-A)^{-1}=\left( z^{\alpha} - \delta^{\alpha}_{\tau}(e^{-z\tau}) \right) \left( \delta^{\alpha}_{\tau}(e^{-z\tau}) -A\right)^{-1} (z^{\alpha}-A)^{-1},$$
we estimate $II$ as following
\begin{equation*}
\|II\|  \leq c \tau^{2} |z|^{2+\alpha}  c |z|^{-\alpha} c |z|^{-\alpha}  |z|^{l} \leq c \tau^{2} |z|^{l+2-\alpha}.
\end{equation*}
According to the triangle inequality,  the desired result is obtained.

Next we consider
\begin{equation*}
\left( \delta^{\alpha}_{\tau}(e^{-z\tau}) -A\right)^{-1} \delta^{l}_{\tau}(e^{-z\tau})  \frac{\gamma_{l}(e^{-z\tau})}{l!}  \tau^{l+1}  - (z^{\alpha}-A)^{-1} z^{-1} = J_{1} + J_{2}
\end{equation*}
with
\begin{equation*}
\begin{split}
J_{1}  & = \left( \delta^{\alpha}_{\tau}(e^{-z\tau}) -A\right)^{-1}  \delta^{l}_{\tau}(e^{-z\tau}) \left[ \frac{\gamma_{l}(e^{-z\tau})}{l!}  \tau^{l+1}   -  z^{-l-1} \right],  \\
J_{2}  & = \left[ \left( \delta^{\alpha}_{\tau}(e^{-z\tau}) -A\right)^{-1} \delta^{l}_{\tau}(e^{-z\tau})   - (z^{\alpha}-A)^{-1} z^{l} \right]z^{-l-1}.
\end{split}
\end{equation*}
According to \eqref{discrete fractional resolvent estimate} and Lemmas \ref{Lemma 2.3}, \ref{Lemma nn3.3} with $l=1, 2$, we obtain
\begin{equation*}
\|J_{1}\|  \leq  c \tau^{l+1} |z|^{l-\alpha} \leq  c \tau^{2} |z|^{1-\alpha} .
\end{equation*}
From  $I$ and $II$, we have
\begin{equation*}
\|J_{2}\|  \leq  c \tau^{2} |z|^{l+2-\alpha} |z|^{-l-1} = c \tau^{2} |z|^{1-\alpha}.
\end{equation*}
According to the triangle inequality,  the desired result is obtained.
\end{proof}

\begin{lemma}\label{addLemma4.5}
Let $\widehat{G}(z) = \frac{1}{\mu+1} \frac{\Gamma(\mu+2)}{z^{\mu+2}} q$ and  $\widetilde{G}(e^{-z\tau}) =q \frac{\tau^{\mu+1}}{\mu+1} Li_{-\mu-1}(e^{-z\tau})$. Then
\begin{equation*}
\left\|\tau \widetilde{G}(e^{-z\tau}) - \widehat{G}(z) \right\| \leq c \tau^{\mu+2}\left\| q \right\|,~~\mu \notin \mathbb{N}.
\end{equation*}
\end{lemma}
\begin{proof}
Using  the definitions of $\widehat{G}(z)$ and $\widetilde{G}(e^{-z\tau})$ and Lemma \ref{addLemma:LipCA} with $p=-\mu-1$, we have
\begin{equation*}
\begin{split}
\left\|\tau \widetilde{G}(e^{-z\tau}) - \widehat{G}(z) \right\|
= &  \left\| \frac{\tau^{\mu+2}}{(\mu+1)} \left( Li_{-\mu-1}(e^{-z\tau})  -  \frac{\Gamma(\mu+2)}{(z\tau)^{\mu+2}} \right) q \right\| \\
\leq  & \frac{\tau^{\mu+2}}{(\mu+1)} \left| \sum_{j=0}^{\infty} (-1)^{j} \zeta(-\mu-1-j)\frac{(z\tau)^{j}}{j!} \right| \left\| q \right\| \leq  c \tau^{\mu+2} \left\| q \right\|.
\end{split}
\end{equation*}
The proof is completed.
\end{proof}

\begin{theorem}[ID1-BDF2]\label{addtheorema4.1}
Let $V(t_{n})$ and $V^{n}$ be the solutions of  \eqref{rrfee} and \eqref{2.3}, respectively. Let $v\in L^{2}(\Omega)$ and $g(x,t)=t^{\mu}q(x) $, $\mu >0$, $q(x) \in L^{2}(\Omega)$.  Then
\begin{equation*}
\left\|V^{n}-V(t_{n})\right\| \leq c \tau^{2}  t^{-2}_{n}\| v \| + c\tau^{\mu + 2} t^{\alpha-2}_{n}\| q \|  + c \tau^{2} t_{n}^{\alpha+\mu-2}\left\| q \right\| .
\end{equation*}
\end{theorem}
\begin{proof}
From Theorem \ref{addtheorema3.1}, the desired results is obtained with $\mu \in \mathbb{N}$. We next prove the case $\mu \notin \mathbb{N}$.
Subtracting \eqref{add4.1} from \eqref{add4.4}, we obtain
\begin{equation*}
V^{n}-V(t_{n})=I_{1}-I_{2} + I_{3} - I_{4}
\end{equation*}
with
\begin{equation*}
\begin{split}
I_{1} =& \frac{1}{2\pi i} \int_{\Gamma^{\tau}_{\theta,\kappa}} e^{zt_{n}}  \left[ \left( \delta^{\alpha}_{\tau}(e^{-z\tau}) -A\right)^{-1} \delta_{\tau}(e^{-z\tau})  \gamma_{1}(e^{-z\tau}) \tau^{2}  - (z^{\alpha}-A)^{-1} z^{-1}\right] Av dz, \\
I_{2} =& \frac{1}{2\pi i} \int_{\Gamma_{\theta,\kappa}\backslash \Gamma^{\tau}_{\theta,\kappa}} e^{zt_{n}}  (z^{\alpha}-A)^{-1} z^{-1}  Av dz,\\
I_{3} =& \frac{1}{2\pi i}\int_{\Gamma^{\tau}_{\theta,\kappa}} e^{zt_{n}} \left[\left( \delta^{\alpha}_{\tau}(e^{-z\tau}) -A\right)^{-1} \delta_{\tau}(e^{-z\tau}) \tau \widetilde{G} (e^{-z\tau})
-   (z^{\alpha}-A)^{-1} z \widehat{G}(z)\right] dz, \\
I_{4} =& \frac{1}{2\pi i} \int_{\Gamma_{\theta,\kappa}\backslash \Gamma^{\tau}_{\theta,\kappa}} e^{zt_n}  (z^{\alpha}-A)^{-1} z  \widehat{G}(z) dz.
\end{split}
\end{equation*}
 According to  \eqref{add3.161} and \eqref{add3.162}, we estimate  $I_{1}$ and $I_{2}$ as following
\begin{equation*}
\left\|I_{1}\right\| \leq c\tau^{2} t_{n}^{-2} \left\| v \right\|
\quad {\rm and} \quad
\left\|I_{2}\right\| \leq c\tau^{2} t_{n}^{-2} \left\| v \right\|.
\end{equation*}
From \eqref{add3.16}, we estimate that $I_{4}$ is similar to $I_{2}$ as following
\begin{equation*}
\begin{split}
\left\|I_{4}\right\|
&\leq c\int_{\Gamma_{\theta,\kappa}\backslash \Gamma^{\tau}_{\theta,\kappa}} {\left|e^{zt_{n}} \right||z|^{-\alpha}\left\|z \widehat{G}(z)\right\| } |dz| \\
&\leq c\int_{\Gamma_{\theta,\kappa}\backslash \Gamma^{\tau}_{\theta,\kappa}} {\left|e^{zt_{n}}\right||z|^{-\alpha} |z|^{-\mu-1} } \left\| q \right\|  |dz|
\leq c\tau^{2} t^{\alpha+\mu-2}_{n}  \left\| q \right\|.
\end{split}
\end{equation*}
Finally we consider $I_{3}= I_{31} + I_{32}$
with
\begin{equation*}
\begin{split}
I_{31} =& \frac{1}{2\pi i}\int_{\Gamma^{\tau}_{\theta,\kappa}} e^{zt_{n}} \left( \delta^{\alpha}_{\tau}(e^{-z\tau}) -A\right)^{-1} \delta_{\tau}(e^{-z\tau})  \left(\tau \widetilde{G} (e^{-z\tau})
-   \widehat{G}(z)\right) dz,\\
I_{32} =& \frac{1}{2\pi i} \int_{\Gamma^{\tau}_{\theta,\kappa}} e^{zt_n}  \left(\left( \delta^{\alpha}_{\tau}(e^{-z\tau}) -A\right)^{-1} \delta_{\tau}(e^{-z\tau})  - (z^{\alpha}-A)^{-1} z \right) \widehat{G}(z) dz.
\end{split}
\end{equation*}
According to \eqref{discrete fractional resolvent estimate} and Lemmas \ref{Lemma 2.3} and \ref{addLemma4.5}, there exists
\begin{equation*}
\left\|I_{31}\right\|
\leq c \tau^{\mu+2}  \left\| q \right\| \int_{\Gamma^{\tau}_{\theta,\kappa}} \left|e^{zt_{n}}\right| |z|^{1-\alpha} |dz|
\leq c \tau^{\mu+2}  t_{n}^{\alpha-2} \left\| q \right\|.
\end{equation*}
From  Lemma \ref{addLemma 4.3} and $\widehat{G}(z) = \frac{1}{\mu+1} \frac{\Gamma(\mu+2)}{z^{\mu+2}} q$, we estimate $I_{32}$ as following
\begin{equation*}
\begin{split}
\left\|I_{32}\right\|
\leq &  c \tau^{2} \left\| q \right\| \int_{\Gamma^{\tau}_{\theta,\kappa}} \left|e^{zt_{n}}\right| |z|^{3-\alpha} |z|^{-\mu-2}  |dz|
\leq    c \tau^{2}  t_{n}^{\alpha+\mu-2}\left\| q \right\|.
\end{split}
\end{equation*}
By the triangle inequality,  the desired result is obtained.
\end{proof}

\subsection{Singular source term}
In this subsection, we consider the singular source term $g(x,t)=t^{\mu}q(x)$ with $ \mu\geq -\alpha$ for subdiffusion \eqref{rrrfee}.

Let $\mathcal{G}(t)=J^{2}g(t) = \frac{t^{\mu+2}}{(\mu+1)(\mu+2)}q$. Using  $\widehat{\mathcal{G}}(z) =  \frac{\Gamma(\mu+1)}{z^{\mu+3}} q$
and  \eqref{LT2}, we have 
\begin{equation}\label{add4.2}
\begin{split}
V(t)
&=  \frac{1}{2\pi i} \int_{\Gamma_{\theta, \kappa}} e^{zt}  (z^{\alpha}-A)^{-1}\left( z^{-1}Av   +  \frac{\Gamma(\mu+1)}{z^{\mu+1}} q \right) dz.
\end{split}
\end{equation}
From  \eqref{DLT2}, it yields
\begin{equation*}
V^{n}=\frac{\tau}{2\pi i}\int_{\Gamma^{\tau}_{\theta,\kappa}} e^{zt_n} \left( \delta^{\alpha}_{\tau}(e^{-z\tau}) -A\right)^{-1} \delta^{2}_{\tau}(e^{-z\tau}) \left( \frac{\gamma_{2}(e^{-z\tau})}{2} \tau^{2} Av + \widetilde{\mathcal{G}} (e^{-z\tau}) \right)dz
\end{equation*}
with $\frac{\gamma_{2}(e^{-z\tau})}{2}= \frac{e^{-z\tau}+e^{-2z\tau}}{2(1-e^{-z\tau})^3}$ and $\Gamma^{\tau}_{\theta, \kappa}=\{z\in \Gamma_{\theta, \kappa}: |\Im z|\leq \pi / \tau\}$.
Here
\begin{equation*}
\widetilde{\mathcal{G}}(\xi)
= \sum^{\infty}_{n=1} \mathcal{G}^{n} \xi^{n}
= q \frac{\tau^{\mu+2}}{(\mu+2)(\mu+1)} \sum^{\infty}_{n=1} \frac{\xi^{n}}{n^{-\mu-2}} = q \frac{\tau^{\mu+2}}{(\mu+2)(\mu+1)} Li_{-\mu-2}(\xi).
\end{equation*}

\begin{lemma}\label{Lemma4.5}
Let $\widehat{\mathcal{G}}(z) = q \frac{\Gamma(\mu+1)}{z^{\mu+3}} $ and  $\widetilde{\mathcal{G}}(e^{-z\tau}) = q \frac{\tau^{\mu+2}}{(\mu+2)(\mu+1)} Li_{-\mu-2}(e^{-z\tau})$. Then
\begin{equation*}
\left\| \tau \widetilde{\mathcal{G}}(e^{-z\tau}) - \widehat{\mathcal{G}}(z) \right\| \leq c \tau^{\mu+3} \left\| q \right\|,~~\mu \notin \mathbb{N}.
\end{equation*}
\end{lemma}
\begin{proof}
From Lemma \ref{addLemma:LipCA}, we have
\begin{equation*}
\begin{split}
\left\| \tau \widetilde{\mathcal{G}}(e^{-z\tau}) - \widehat{\mathcal{G}}(z) \right\|
= & \left\| \frac{\tau^{\mu+3}}{(\mu+2)(\mu+1)} \left(Li_{-\mu-2}(e^{-z\tau})  -  \frac{\Gamma(\mu+3)}{(z\tau)^{\mu+3}} \right) q \right\|\\
\leq & \frac{\tau^{\mu+3}}{(\mu+2)(\mu+1)} \left| \sum_{j=0}^{\infty} (-1)^{j} \zeta(-\mu-2-j)\frac{(z\tau)^{j}}{j!} \right| \left\| q \right\| \\
\leq & c \tau^{\mu+3} \left\| q \right\|.
\end{split}
\end{equation*}
The proof is completed.
\end{proof}
\begin{theorem}[ID2-BDF2]\label{theorema4.1}
Let $V(t_{n})$ and $V^{n}$ be the solutions of \eqref{rrrfee} and \eqref{4.1}, respectively. Let $v\in L^{2}(\Omega)$ and  $g(x,t)=t^{\mu} q(x)$, $\mu \geq -\alpha$,  $q(x) \in L^{2}(\Omega)$.  Then
\begin{equation*}
\left\|V^{n}-V(t_{n})\right\| \leq c \tau^{2}  t_{n}^{-2} \| v \| + c\tau^{\mu+3}  t_{n}^{\alpha-3} \left\| q \right\|  + c \tau^{2} t_{n}^{\alpha+\mu-2}\left\| q \right\| .
\end{equation*}
\end{theorem}
\begin{proof}
From Theorem \ref{addtheorema3.1}, the desired results is obtained with $\mu \in \mathbb{N}$. We next prove the case $\mu \notin \mathbb{N}$.
Subtracting \eqref{LT2} from \eqref{DLT2}, we obtain
\begin{equation*}
V^{n}-V(t_{n})=I_{1}-I_{2} + I_{3} - I_{4}
\end{equation*}
with
\begin{equation*}
\begin{split}
I_{1} =& \frac{1}{2\pi i} \int_{\Gamma^{\tau}_{\theta,\kappa}} e^{zt_{n}}  \left[ \left( \delta^{\alpha}_{\tau}(e^{-z\tau}) -A\right)^{-1} \delta^{2}_{\tau}(e^{-z\tau})  \frac{e^{-z\tau}+e^{-2z\tau}}{2(1-e^{-z\tau})^3} \tau^{3}  - (z^{\alpha}-A)^{-1} z^{-1}\right] Av dz, \\
I_{2} =& \frac{1}{2\pi i} \int_{\Gamma_{\theta,\kappa}\backslash \Gamma^{\tau}_{\theta,\kappa}} e^{zt_{n}}  (z^{\alpha}-A)^{-1} z^{-1}  Av dz,\\
I_{3} =& \frac{1}{2\pi i}\int_{\Gamma^{\tau}_{\theta,\kappa}} e^{zt_{n}} \left[\left( \delta^{\alpha}_{\tau}(e^{-z\tau}) -A\right)^{-1} \delta^{2}_{\tau}(e^{-z\tau}) \tau \widetilde{\mathcal{G}} (e^{-z\tau})
-   (z^{\alpha}-A)^{-1} z^{2} \widehat{\mathcal{G}}(z)\right] dz, \\
I_{4} =& \frac{1}{2\pi i} \int_{\Gamma_{\theta,\kappa}\backslash \Gamma^{\tau}_{\theta,\kappa}} e^{zt_n}  (z^{\alpha}-A)^{-1} z^{2} \widehat{\mathcal{G}}(z) dz.
\end{split}
\end{equation*}
Using  \eqref{add3.161}, \eqref{add3.162} and Lemma \ref{addLemma 4.3}, we estimate  $I_{1}$ and $I_{2}$ as following
\begin{equation*}
\left\|I_{1}\right\|  \leq c\tau^{2} t_{n}^{-2} \left\| v \right\| \quad {\rm and} \quad
\left\|I_{2}\right\|  \leq c\tau^{2} t_{n}^{-2} \left\| v \right\|.
\end{equation*}
By \eqref{add3.16}, we estimate that $I_{4}$ is similar to $I_{2}$ as following
\begin{equation*}
\begin{split}
\left\|I_{4}\right\|
&\leq c\int_{\Gamma_{\theta,\kappa}\backslash \Gamma^{\tau}_{\theta,\kappa}} {\left|e^{zt_{n}} \right||z|^{-\alpha}\left\|z^{2} \widehat{\mathcal{G}}(z)\right\| } |dz| \\
&\leq c \left\| q \right\| \int_{\Gamma_{\theta,\kappa}\backslash \Gamma^{\tau}_{\theta,\kappa}} {\left|e^{zt_{n}}\right||z|^{-\alpha} |z|^{-\mu-1} }  |dz| \leq c\tau^{2} t^{\alpha+\mu-2}_{n} \left\| q \right\|.
\end{split}
\end{equation*}
Finally we consider $I_{3}= I_{31} + I_{32}$ with
\begin{equation*}
\begin{split}
I_{31} =& \frac{1}{2\pi i}\int_{\Gamma^{\tau}_{\theta,\kappa}} e^{zt_{n}} \left( \delta^{\alpha}_{\tau}(e^{-z\tau}) -A\right)^{-1} \delta^{2}_{\tau}(e^{-z\tau})  \left(\tau \widetilde{\mathcal{G}} (e^{-z\tau})
-   \widehat{\mathcal{G}}(z)\right) dz, \\
I_{32} =& \frac{1}{2\pi i} \int_{\Gamma^{\tau}_{\theta,\kappa}} e^{zt_n}  \left(\left( \delta^{\alpha}_{\tau}(e^{-z\tau}) -A\right)^{-1} \delta^{2}_{\tau}(e^{-z\tau})  - (z^{\alpha}-A)^{-1} z^{2} \right) \widehat{\mathcal{G}}(z) dz.
\end{split}
\end{equation*}
According to \eqref{discrete fractional resolvent estimate} and Lemmas \ref{Lemma 2.3} and \ref{Lemma4.5}, there exists
\begin{equation*}
\left\|I_{31}\right\|
\leq c \tau^{\mu+3} \left\| q \right\|  \int_{\Gamma^{\tau}_{\theta,\kappa}} \left|e^{z t_{n}}\right| |z|^{2-\alpha} |dz|
\leq c\tau^{\mu+3}  t_{n}^{\alpha-3} \left\| q \right\|.
\end{equation*}
From Lemma \ref{addLemma 4.3}, we estimate $I_{32}$ as following
\begin{equation*}
\begin{split}
\left\|I_{32}\right\|
\leq &  c  \tau^{2} \left\| q \right\| \int_{\Gamma^{\tau}_{\theta,\kappa}} \left|e^{z t_{n}}\right| |z|^{4-\alpha} |z|^{-\mu-3}  |dz| \\
\leq &   c \tau^{2} \left\| q \right\|  \int_{\Gamma^{\tau}_{\theta,\kappa}} \left|e^{z t_{n}}\right|  |z|^{1-\alpha-\mu} |dz|
\leq c\tau^{2}  t_{n}^{\alpha+\mu-2}\left\| q \right\|.
\end{split}
\end{equation*}
By the triangle inequality,  the desired result is obtained.
\end{proof}

\section{Convergence analysis: Source function $t^{\mu}\circ f(x,t)$ with $\mu>-1$}
Based on the  discussion  of  Section 3 and 4, we now analyse  the error estimates for subdiffusion \eqref{fee}  with the singular  source term  $t^{\mu}\circ f(x,t)$.
\subsection{Convergence analysis: Convolution source function $t^{\mu} \ast f(t)$, $\mu>-1$}
Let $f(t) = f(0) + tf'(0) + t\ast f''(t)$.
Then  we obtain
\begin{equation*}
g(t) =  t^{\mu} \ast f(t) =  \frac{t^{\mu+1}  f(0)}{\mu+1}   + \frac{t^{\mu+2}  f'(0)}{(\mu+1)(\mu+2)} +  t^{\mu} \ast t\ast f''(t).
\end{equation*}
Let $G(t) = J^1g(t)=\frac{1}{\mu+1} t^{\mu+1} \ast f(t)$ with $G(0)=0$.  It yields
\begin{equation*}
\begin{split}
G(t) & 
=  \frac{t^{\mu+2}  f(0)}{(\mu+1)(\mu+2)}   + \frac{t^{\mu+3}  f'(0)}{(\mu+1)(\mu+2)(\mu+3)} + \frac{1}{\mu+1} t^{\mu+1} \ast t \ast f''(t) \\
& =  \frac{t^{\mu+2}  f(0)}{(\mu+1)(\mu+2)}   + \frac{t^{\mu+3}  f'(0)}{(\mu+1)(\mu+2)(\mu+3)} +  \frac{t^{2}}{2}  \ast \left( t^{\mu} \ast f''(t)\right),
\end{split}
\end{equation*}
where we use
\begin{equation*}
t^{\mu+1} \ast t =  \int_{0}^{t} (t-s)^{\mu+1}  s ds
= \frac{\mu+1}{2}\int_{0}^{t} (t-s)^{\mu}  s^2 ds = \frac{\mu+1}{2} t^{2} \ast t^{\mu}.
\end{equation*}

\begin{lemma}\label{addLemma4.6}
Let $V(t_{n})$ and $V^{n}$ be the solutions of \eqref{rrfee} and \eqref{2.3}, respectively.
Let $v=0$, $G(t):=  \frac{t^2}{2} \ast \left( t^{\mu} \ast f''(t)\right)$ with $\mu>-1$ and $\int_{0}^{t}  (t - s)^{\alpha-1}  s^{\mu} \ast \left\| f''(s) \right\|ds<\infty $. Then
\begin{equation*}
\left\|V(t_{n})-V^{n}\right\| \leq c\tau^{2} \int_{0}^{t_{n}}  (t_{n} - s)^{\alpha-1}  s^{\mu} \ast \left\| f''(s) \right\|ds \leq c\tau^{2} \int_{0}^{t_{n}} (t_{n} - s)^{\alpha+\mu} \left\| f''(s) \right\|ds.
\end{equation*}
\end{lemma}
\begin{proof}
By Lemma \ref{lemma3.10} with $g''(t) =t^{\mu} \ast f''(t)$, we obtain 
\begin{equation*}
\begin{split}
\left\|V(t_{n})-V^{n}\right\|
&\leq c\tau^{2} \int_{0}^{t_{n}} (t_n-s)^{\alpha-1} \left\|  s^{\mu} \ast  f''(s)  \right\|ds \\
& 
\leq c\tau^{2} \int_{0}^{t_{n}}  (t_{n} - s)^{\alpha-1}  s^{\mu} \ast \left\| f''(s) \right\|ds\\
&= c\tau^{2}\left(t^{\alpha-1}\ast t^{\mu}\right) \ast \left\| f''(t) \right\|_{t=t_n}
  \leq c\tau^{2} \int_{0}^{t_{n}}  (t_{n} - s)^{\alpha+\mu} \left\| f''(s) \right\|ds .
\end{split}
\end{equation*}
The proof is completed.
\end{proof}
\begin{theorem}[ID1-BDF2]\label{addtheorema4.2}
Let $V(t_{n})$ and $V^{n}$ be the solutions of \eqref{rrfee} and \eqref{2.3}, respectively. Let $v\in L^{2}(\Omega)$, $g(t)=t^{\mu} \ast f(t) $ with $\mu>-1$ and  $f \in C^{1}([0,T]; L^{2}(\Omega))$,
$\int_{0}^{t}  (t - s)^{\alpha-1}  s^{\mu} \ast \left\| f''(s) \right\|ds<\infty $.  Then
\begin{equation*}
\begin{split}
&\left\|V^{n}-V(t_{n})\right\|\\
&\leq c\tau^{2}   \left( t^{-2}_{n} \|v\|    + t_{n}^{\alpha + \mu-1}\left\| f(0) \right\|   + t_{n}^{\alpha + \mu} \left\| f'(0) \right\|  +  \int_{0}^{t_{n}}  (t_{n} - s)^{\alpha-1}  s^{\mu} \ast \left\| f''(s) \right\|ds \right)\\
&\leq c\tau^{2}   \left( t^{-2}_{n} \|v\|    + t_{n}^{\alpha + \mu-1}\left\| f(0) \right\|   + t_{n}^{\alpha + \mu} \left\| f'(0) \right\|  +  \int_{0}^{t_{n}} (t_{n} - s)^{\alpha+\mu} \left\| f''(s) \right\|ds \right).
\end{split}
\end{equation*}
\end{theorem}

\begin{proof}
According to Theorem \ref{addtheorema4.1},  Lemma \ref{addLemma4.6}, and similar treatment of the initial data $v$ in Theorem  \ref{addtheorema3.1},   the desired result is obtained.
\end{proof}


\subsection{Convergence analysis: product source function $t^{\mu}  f(t)$, $\mu>0$}
Let $G(t) = J^1g(t)$ and  $f(t) = f(0) + tf'(0) + t\ast f''(t)$.
Then we have
\begin{equation*}
G(t) = 1 \ast ( t^{\mu} f(t)) =  \frac{t^{\mu+1}  f(0)}{\mu+1}   + \frac{t^{\mu+2}  f'(0)}{\mu+2} + 1 \ast \left[ t^{\mu} \left( t \ast  f''(t) \right) \right].
\end{equation*}
Let $h(t)= t^{\mu} \left( t \ast  f''(t) \right) $ with $h(0)=0$. It leads to
\begin{equation*}
h'(t) = \mu t^{\mu-1} \left( t \ast  f''(t) \right) + t^{\mu} \left( 1 \ast  f''(t) \right)
\end{equation*}
with $h'(0)=0$, since
\begin{equation*}
\begin{split}
\left| h'(t) \right|
&\leq \left| \mu t^{\mu-1} \int_{0}^{t} (t-s)   f''(s) ds \right| +  \left| t^{\mu} \int_{0}^{t}   f''(s) ds \right| \leq (\mu+1) t^{\mu} \int_{0}^{t}  \left| f''(s) \right| ds,~\mu>0.
\end{split}
\end{equation*}
Moreover, there exists
\begin{equation}\label{addeq5.5}
h''(t) = \mu \left( \mu-1 \right) t^{\mu-2} \left( t \ast  f''(t) \right) + 2 \mu t^{\mu-1} \left( 1 \ast  f''(t) \right) + t^{\mu}   f''(t).
\end{equation}
Thus one has
\begin{equation}\label{eq5.5}
1 \ast h(t) = t h(0) + \frac{t^2}{2} h'(0) + \frac{t^2}{2} \ast h''(t) = \frac{t^2}{2} \ast h''(t).
\end{equation}

\begin{lemma}\label{Lemma5.7}
Let $V(t_{n})$ and $V^{n}$ be the solutions of \eqref{rrfee} and \eqref{2.3}, respectively.
Let $v=0$, $G(t)= 1 \ast \left[ t^{\mu} \left( t \ast  f''(t) \right) \right]$ with $\mu>0$ and $f \in C^{1}([0,T]; L^{2}(\Omega))$,
$\int_{0}^{t} \left\| f''(s) \right\| ds <\infty$,    $\int_{0}^{t} (t-s)^{\alpha-1} s^{\mu}  \left\| f''(s) \right\|ds<\infty$. Then
\begin{equation*}
\left\|V(t_{n})-V^{n}\right\| \leq c\tau^{2} \left( t_n^{\alpha + \mu -1} \int_{0}^{t_{n}}   \left\| f''(s) \right\|  ds + \int_{0}^{t_{n}} (t_n-s)^{\alpha-1} s^{\mu}  \left\| f''(s) \right\|ds \right).
\end{equation*}
\end{lemma}
\begin{proof}
Let $h(t)= t^{\mu} \left( t \ast  f''(t) \right)$.
From  \eqref{eq5.5}, we have $G(t)= 1 \ast h(t) = \frac{t^2}{2} \ast h''(t)$.
According to Lemma \ref{lemma3.10}  and \eqref{addeq5.5}, it yields  
\begin{equation*}
\left\|V(t_{n})-V^{n}\right\|
\leq c\tau^{2} \int_{0}^{t_{n}} (t_n-s)^{\alpha-1} \left\|  h''(s) \right\|ds  \leq c\tau^{2} \left( I_{1} + I_{2} + I_{3} \right)
\end{equation*}
with
\begin{equation*}
\begin{split}
I_{1} & = \int_{0}^{t_{n}} (t_n-s)^{\alpha-1} \left\|   s^{\mu-2} \left( s \ast  f''(s) \right) \right\|ds, \\
I_{2} & = \int_{0}^{t_{n}} (t_n-s)^{\alpha-1} \left\|  s^{\mu-1} \left( 1 \ast  f''(s) \right) \right\|ds
~~{\rm and}~~I_{3}  = \int_{0}^{t_{n}} (t_n-s)^{\alpha-1} \left\|  s^{\mu}   f''(s) \right\|ds.
\end{split}
\end{equation*}
We  estimate  $I_{1}$  as following
\begin{equation*}
\begin{split}
I_{1}
& =   \int_{0}^{t_{n}} (t_n-s)^{\alpha-1} s^{\mu-1} \left\|  \int_{0}^{s} \frac{s-w}{s}   f''(w) dw \right\|ds \\
& \leq   \int_{0}^{t_{n}} (t_n-s)^{\alpha-1} s^{\mu-1} \int_{0}^{t_{n}}   \left\| f''(w) \right\| dw ds
 =  B(\alpha, \mu)  t_n^{\alpha + \mu -1} \int_{0}^{t_{n}}   \left\| f''(w) \right\|  dw,
\end{split}
\end{equation*}
since
\begin{equation*}
\int_{0}^{t_{n}} (t_n-s)^{\alpha-1} s^{\mu-1} ds = t_n^{\alpha + \mu -1} \int_{0}^{1} (1-s)^{\alpha-1} s^{\mu-1} ds = B(\alpha, \mu) t_n^{\alpha + \mu -1}.
\end{equation*}
Similarly, we  estimate  $I_{2}$  as following
\begin{equation*}
I_{2} \leq   \int_{0}^{t_{n}} (t_n-s)^{\alpha-1} s^{\mu-1} \int_{0}^{t_{n}}   \left\| f''(w) \right\| dw ds
 =  B(\alpha, \mu)  t_n^{\alpha + \mu -1} \int_{0}^{t_{n}}   \left\| f''(w) \right\|  dw.
\end{equation*}
By the triangle inequality, we obtain
\begin{equation*}
\left\|V(t_{n})-V^{n}\right\| \leq c\tau^{2} \left( t_n^{\alpha + \mu -1} \int_{0}^{t_{n}}   \left\| f''(s) \right\|  ds + \int_{0}^{t_{n}} (t_n-s)^{\alpha-1} s^{\mu}  \left\| f''(s) \right\|ds \right).
\end{equation*}
The proof is completed.
\end{proof}
\begin{theorem}[ID1-BDF2]\label{Theorem5.8}
Let $V(t_{n})$ and $V^{n}$ be the solutions of \eqref{rrfee} and \eqref{2.3}, respectively. Let $v\in L^{2}(\Omega)$, $g(t)= t^{\mu}  f(t)$ with $\mu>0$ and $f \in C^{1}([0,T]; L^{2}(\Omega))$,  $\int_{0}^{t} \left\| f''(s) \right\| ds <\infty$,  $\int_{0}^{t} (t-s)^{\alpha-1} s^{\mu}  \left\| f''(s) \right\|ds<\infty$. Then
\begin{equation*}
\begin{split}
\left\|V^{n}-V(t_{n})\right\|
& \leq c\tau^{2}   \left( t^{-2}_{n} \|v\|    + t_{n}^{\alpha+ \mu -2} \left\| f(0) \right\|   + t_{n}^{\alpha+ \mu -1} \left\| f'(0) \right\| \right) \\
& \quad + c\tau^{2} \left( t_n^{\alpha + \mu -1} \int_{0}^{t_{n}}   \left\| f''(s) \right\|  ds + \int_{0}^{t_{n}} (t_n-s)^{\alpha-1} s^{\mu}  \left\| f''(s) \right\|ds \right).
\end{split}
\end{equation*}
\end{theorem}
\begin{proof}
According to Theorem \ref{addtheorema4.1},  Lemma \ref{Lemma5.7}, and similar treatment of the initial data $v$ in Theorem  \ref{addtheorema3.1},   the desired result is obtained.
\end{proof}

\subsection{Convergence analysis: product source function $t^{\mu}  f(t)$, $ -\alpha\leq \mu<0$}
Let $\mathcal{G}(t) = J^{2}g(t)$ and  $f(t) = f(0) + tf'(0) + t\ast f''(t)$.
Then we have
\begin{equation*}
\mathcal{G}(t) = t \ast \left( t^{\mu} f(t) \right) = \frac{t^{\mu+2} f(0)}{(\mu+1)(\mu+2)} + \frac{t^{\mu+3} f'(0)}{(\mu+2)(\mu+3)} + t \ast \left[ t^{\mu} \left( t \ast f''(t)\right) \right].
\end{equation*}
Let $h(t)= t^{\mu} \left( t \ast  f''(t) \right) $ with $h(0)=0$. It leads to
\begin{equation*}
h'(t) = \mu t^{\mu-1} \left( t \ast  f''(t) \right) + t^{\mu} \left( 1 \ast  f''(t) \right),
\end{equation*}
which implies
\begin{equation*}
\left| h'(0) \right| \leq (\mu+1) \int_{0}^{t} s^{\mu} \left| f''(s) \right| ds,
\end{equation*}
since
\begin{equation*}
\begin{split}
\left| h'(t) \right|
&   \leq (\mu+1) t^{\mu} \int_{0}^{t}  \left| f''(s) \right| ds  \leq (\mu+1) \int_{0}^{t} s^{\mu} \left| f''(s) \right| ds ~~{\rm with}~ -1<\mu<0.
\end{split}
\end{equation*}
Thus we get
\begin{equation}\label{addnneq5.5}
t \ast h(t) = \frac{t^{2}}{2} h(0) + \frac{t^3}{6} h'(0) + \frac{t^3}{6} \ast h''(t) = \frac{t^3}{6} h'(0) + \frac{t^3}{6} \ast h''(t).
\end{equation}

\begin{lemma}\label{addnnLemma5.7}
Let $V(t_{n})$ and $V^{n}$ be the solutions of \eqref{rrrfee} and \eqref{4.1}, respectively.
Let $v=0$, $\mathcal{G}(t)= t \ast \left[ t^{\mu} \left( t \ast  f''(t) \right) \right]$ with $-\alpha\leq \mu<0$ and
$f \in C^{1}([0,T]; L^{2}(\Omega))$,  $\int_{0}^{t} s^{\frac{\mu-1}{2}} \left\| f''(s) \right\| ds <\infty$,  $\int_{0}^{t} (t-s)^{\alpha-1} s^{\mu}  \left\| f''(s) \right\|ds$. Then
\begin{equation*}
\left\|V(t_{n})-V^{n}\right\| \leq
c\tau^{2} \left( t_n^{\alpha + \frac{\mu -1}{2}} \int_{0}^{t_{n}}  s^{\frac{\mu -1}{2}} \left\| f''(s) \right\|  ds + \int_{0}^{t_{n}} (t_n-s)^{\alpha-1} s^{\mu} \left\| f''(s) \right\|ds \right).
\end{equation*}
\end{lemma}
\begin{proof}
Let  $h(t)= t^{\mu} \left( t \ast  f''(t) \right)$. From  \eqref{addnneq5.5}, we have
$$\mathcal{G}(t)= t \ast h(t) = \frac{t^3}{6} h'(0) + \frac{t^3}{6} \ast h''(t).$$
According to Theorems \ref{theorema4.1}, \ref{addtheorema3.2}  and \eqref{addeq5.5}, it yields  
\begin{equation*}
\begin{split}
\left\|V(t_{n})-V^{n}\right\|
& \leq c \tau^{2} \left( t_{n}^{\alpha-1}\left\| h'(0) \right\| +  \int_{0}^{t_{n}} (t_n-s)^{\alpha-1} \left\|  h''(s) \right\|ds \right) \\
& \leq c\tau^{2} \left( I_{1} + I_{2} + I_{3} + I_{4} \right)
\end{split}
\end{equation*}
with
\begin{equation*}
\begin{split}
I_{1} & =  t_{n}^{\alpha-1}\left\| h'(0) \right\| , \quad
I_{2}   = \int_{0}^{t_{n}} (t_n-s)^{\alpha-1} \left\|   s^{\mu-2} \left( s \ast  f''(s) \right) \right\|ds, \\
I_{3} & = \int_{0}^{t_{n}} (t_n-s)^{\alpha-1} \left\|  s^{\mu-1} \left( 1 \ast  f''(s) \right) \right\|ds ~~{\rm and}~~
I_{4}   = \int_{0}^{t_{n}} (t_n-s)^{\alpha-1} \left\|  s^{\mu}   f''(s) \right\|ds.
\end{split}
\end{equation*}
Since
\begin{equation*}
\begin{split}
I_{1}  =  t_{n}^{\alpha-1}\left\| h'(0) \right\| \leq c t_{n}^{\alpha-1}  \int_{0}^{t_{n}} s^{\mu} \left\| f''(s) \right\| ds
\leq c   \int_{0}^{t_{n}} (t_{n}-s)^{\alpha-1}  s^{\mu} \left\| f''(s) \right\| ds,
\end{split}
\end{equation*}
and
\begin{equation*}
\begin{split}
I_{2}
& =   \int_{0}^{t_{n}} (t_n-s)^{\alpha-1} s^{\frac{\mu-1}{2}} \left\|  \int_{0}^{s} \frac{s-w}{s} s^{\frac{\mu-1}{2}}  f''(w) dw \right\|ds \\
& \leq   \int_{0}^{t_{n}} (t_n-s)^{\alpha-1} s^{\frac{\mu-1}{2}} \int_{0}^{t_{n}}  w^{\frac{\mu-1}{2}} \left\| f''(w) \right\| dw ds
\leq  c  t_n^{\alpha + \frac{\mu-1}{2}} \int_{0}^{t_{n}} w^{\frac{\mu-1}{2}}  \left\| f''(w) \right\|  dw,
\end{split}
\end{equation*}
where we use
\begin{equation*}
\int_{0}^{t_{n}} (t_n-s)^{\alpha-1} s^{\frac{\mu-1}{2}} ds = t_n^{\alpha + \frac{\mu -1}{2}} \int_{0}^{1} (1-s)^{\alpha-1} s^{\frac{\mu-1}{2}} ds
 = B\left(\alpha, \frac{\mu+1}{2}\right) t_n^{\alpha + \frac{\mu -1}{2}}.
\end{equation*}
Similarly, we  estimate  $I_{3}$  as following
\begin{equation*}
I_{3} \leq   \int_{0}^{t_{n}} (t_n-s)^{\alpha-1} s^{\mu-1} \int_{0}^{t_{n}}   \left\| f''(w) \right\| dw ds
\leq  c  t_n^{\alpha + \frac{\mu-1}{2}} \int_{0}^{t_{n}} w^{\frac{\mu-1}{2}}  \left\| f''(w) \right\|  dw.
\end{equation*}
By the triangle inequality, we obtain
\begin{equation*}
\left\|V(t_{n})-V^{n}\right\| \leq
c\tau^{2} \left( t_n^{\alpha + \frac{\mu -1}{2}} \int_{0}^{t_{n}}  s^{\frac{\mu -1}{2}} \left\| f''(s) \right\|  ds + \int_{0}^{t_{n}} (t_n-s)^{\alpha-1} s^{\mu} \left\| f''(s) \right\|ds \right).
\end{equation*}
The proof is completed.
\end{proof}

\begin{theorem}[ID2-BDF2]\label{Theorem5.8888}
Let $V(t_{n})$ and $V^{n}$ be the solutions of \eqref{rrrfee} and \eqref{4.1}, respectively.
Let $v\in L^{2}(\Omega)$, $g(t)= t^{\mu}  f(t)$ with $-\alpha\leq \mu<0$ and  $f \in C^{1}([0,T]; L^{2}(\Omega))$,
$\int_{0}^{t}s^{\frac{\mu -1}{2}} \left\| f''(s) \right\| ds <\infty$, $\int_{0}^{t} (t-s)^{\alpha-1} s^{\mu}  \left\| f''(s) \right\|ds$. Then
\begin{equation*}
\begin{split}
\left\|V^{n}-V(t_{n})\right\|
& \leq c\tau^{2}   \left( t^{-2}_{n} \|v\|  + t_{n}^{\alpha+ \mu -2} \left\| f(0) \right\| + t_{n}^{\alpha+ \mu -1} \left\| f'(0) \right\| \right) \\
&  + c\tau^{2} \left( t_n^{\alpha + \frac{\mu -1}{2}} \int_{0}^{t_{n}}  s^{\frac{\mu -1}{2}} \left\| f''(s) \right\|  ds + \int_{0}^{t_{n}} (t_n-s)^{\alpha-1} s^{\mu} \left\| f''(s) \right\|ds \right).
\end{split}
\end{equation*}
\end{theorem}
\begin{proof}
According to Theorem \ref{theorema4.1},  Lemma \ref{addnnLemma5.7}, and similar treatment of the initial data $v$ in Theorem  \ref{addtheorema3.1},   the desired result is obtained.
\end{proof}
\begin{remark}
Theorems  \ref{theorema4.1} and \ref{Theorem5.8888} are naturally extended to $ \mu > -1.$
\end{remark}

\section{Numerical results}\label{Se:numer}
We  numerically verify the above theoretical results  and the discrete
$L^2$-norm is used to measure the numerical errors.
In the space direction, it is discretized with the   spectral collocation method  with the Chebyshev-Gauss-Lobatto points \cite{STW:2011}.
Here we main focus on the time direction convergence order, since the convergence rate of the spatial discretization is well understood.
Since the analytic solutions is unknown, the order of the convergence of the numerical results are computed by the following formula
\begin{equation*}
  {\rm Convergence ~Rate}=\frac{\ln \left(||u^{N/2}-u^{N}||/||u^{N}-u^{2N}||\right)}{\ln 2}
\end{equation*}
with $u^N=V^N+v$ in  \eqref{2.3}.

In the  experiment,  several algorithms including the correction    BDF2 methods \cite{JLZ:2017}  are carried  out and compared with IDk-mehtod:
\begin{equation}\label{s2.3}
 {\rm BDF2~ Method:}~~~~  \partial^{\alpha}_{\tau} V^{n} - AV^{n}=  Av + g^{n}.
\end{equation}
\begin{equation}\label{corbdf2}
~~~~~{\rm Corr\!-\!BDF2~ Method:}~~~~  \partial^{\alpha}_{\tau} V^{n} - AV^{n}= \frac{3}{2} Av + \frac{1}{2} g^{0} + g^{n} .
\end{equation}

\begin{example}\label{EX2}
Let $T=1$ and $\Omega=(-1, 1)$. Consider subdiffusion \eqref{fee} with
$$v(x)=\sin(x)\sqrt{1-x^2}~~{\rm  and}~~  g(x,t)= (1+t^{\mu}+t^{\alpha\mu})\circ (1-t)^{\beta} e^x (1+\chi_{(0,1)}(x)).$$
\end{example}

Here $J^kg(x,t)=t^{k-1}\ast g(x,t)$, $k=1,2$  are calculated by JacobiGL Algorithm \cite{ChD:15,Hesthaven:07},
which is generating the nodes and weights of Gauss-Labatto integral with  the weighting function such as $(1-t)^\mu$ or $(1+t)^\mu$.

\begin{table}[!ht]
{\small \begin{center}
\caption{The discrete $L^2$-norm $||u^{N}-u^{2N}||$ and convergent order of schemes \eqref{s2.3}, \eqref{corbdf2} and \eqref{2.3}, \eqref{4.1} with $\beta=0$, $\alpha=0.7$.
Here $\circ$ denotes the dot product.}
\begin{tabular}{l r l l l l l l}
\hline
Scheme                          &  $\mu$    &  $N=50$      &  $N=100$    & $N=200$     & $N=400$     &  $N=800$      \\ \hline
    \multirow{4}{*}{BDF2}       &  0.8      &  2.4743e-03  & 1.1981e-03  & 5.8732e-04  & 2.9005e-04  & 1.4390e-04    \\
                                &           &              & 1.0462      & 1.0286      & 1.0178      & 1.0113        \\
                                &  -0.8     &  1.5948e-01  & 1.3256e-01  & 1.1109e-01  & 9.3707e-02  & 7.9450e-02    \\
                                &           &              & 0.26679     & 0.25489     & 0.24549     & 0.23811       \\ \hline
   \multirow{3}{*}{Corr-BDF2}  &  0.8      &  9.4381e-05  & 3.6107e-05  & 1.3189e-05  & 4.6888e-06  & 1.6386e-06    \\
                                &           &              & 1.3862      & 1.4529      & 1.4921      & 1.5168            \\
                                &  -0.8     &  NaN         & NaN         & NaN         & NaN         & NaN           \\ \hline
   \multirow{4}{*}{ID1-BDF2}    &  0.8      &  1.6660e-04  & 4.1216e-05  & 1.0249e-05  & 2.5553e-06  & 6.3792e-07    \\
                                &           &              & 2.0151      & 2.0077      & 2.0040      & 2.0021        \\
                                &  -0.8     &  6.7744e-03  & 3.0380e-03  & 1.3367e-03  & 5.8281e-04  & 2.5299e-04    \\
                                &           &              & 1.1570      & 1.1844      & 1.1976      & 1.2039        \\ \hline
   \multirow{4}{*}{ID2-BDF2}    &  0.8      &  3.2389e-04  & 7.9995e-05  & 1.9879e-05  & 4.9539e-06  & 1.2374e-06    \\
                                &           &              & 2.0175      & 2.0087      & 2.0046      & 2.0013        \\
                                &  -0.8     &  2.1611e-03  & 5.2769e-04  & 1.3018e-04  & 3.2292e-05  & 8.0280e-06    \\
                                &           &              & 2.0340      & 2.0192      & 2.0112      & 2.0081        \\ \hline
\end{tabular}\label{table:6.1}
\end{center}}
\end{table}

\begin{table}[!ht]
{\small \begin{center}
\caption{The discrete $L^2$-norm  $||u^{N}-u^{2N}||$ and convergent order of schemes  \eqref{2.3} and \eqref{4.1} with $\beta=1.9$, respectively.
Here $\circ$ denotes the dot product.}
\begin{tabular}{l l r l l l l l l}
\hline
Scheme                          &  $\alpha$             &  $\mu$    &  $N=50$      &  $N=100$    & $N=200$     &  $N=400$    &  $N=800$    \\ \hline
  \multirow{8}{*}{ID1-BDF2}     & \multirow{4}{*}{0.3}  &  0.5      &  1.5025e-03  & 3.9778e-04  & 1.0433e-04  & 2.7198e-05  & 7.0660e-06  \\
                                &                       &           &              & 1.9174      & 1.9307      & 1.9396      & 1.9445      \\
                                &                       &  -0.9     &  4.9903e-03  & 2.7664e-03  & 1.4020e-03  & 6.8259e-04  & 3.2574e-04  \\
                                &                       &           &              & 0.85109     & 0.98050     & 1.0384      & 1.0673      \\  \cline{2-8}
                                & \multirow{4}{*}{0.7}  &  0.5      &  6.8462e-04  & 1.8033e-04  & 4.6484e-05  & 1.1840e-05  & 2.9948e-06  \\
                                &                       &           &              & 1.9247      & 1.9558      & 1.9731      & 1.9831      \\
                                &                       &  -0.9     &  2.0722e-02  & 1.0219e-02  & 4.8849e-03  & 2.3017e-03  & 1.0770e-03  \\
                                &                       &           &              & 1.0199      & 1.0648      & 1.0856      & 1.0956      \\ \hline
  \multirow{8}{*}{ID2-BDF2}     & \multirow{4}{*}{0.3}  &  0.5      &  3.1810e-03  & 8.4340e-04  & 2.2164e-04  & 5.7938e-05  & 1.5180e-05  \\
                                &                       &           &              & 1.9152      & 1.9280      & 1.9356      & 1.9323      \\
                                &                       &  -0.9     &  4.6179e-03  & 1.1806e-03  & 3.0298e-04  & 7.7857e-05  & 2.0182e-05  \\
                                &                       &           &              & 1.9677      & 1.9622      & 1.9603      & 1.9478      \\  \cline{2-8}
                                &  \multirow{4}{*}{0.7} &  0.5      &  1.9266e-03  & 5.0536e-04  & 1.3015e-04  & 3.3167e-05  & 8.4027e-06  \\
                                &                       &           &              & 1.9307      & 1.9571      & 1.9724      & 1.9808      \\
                                &                       &  -0.9     &  7.2846e-03  & 1.8010e-03  & 4.4808e-04  & 1.1179e-04  & 2.7922e-05  \\
                                &                       &           &              & 2.0161      & 2.0070      & 2.0030      & 2.0013      \\ \hline
\end{tabular}\label{table:Ex13}
\end{center}}
\end{table}

\begin{table}[!ht]
{\small\begin{center}
\caption{The discrete $L^2$-norm  $||u^{N}-u^{2N}||$ and convergent order of schemes \eqref{s2.3} and \eqref{2.3} with $\beta=1.9$, respectively.
Here $\circ$ denotes the Laplace convolution.}
\begin{tabular}{c c c c c c c c c}
\hline
Scheme                          &  $\alpha$             &  $\mu$    &  $N=50$      &  $N=100$    & $N=200$     &  $N=400$    &  $N=800$     \\ \hline
  \multirow{8}{*}{ID1-BDF2}     & \multirow{4}{*}{0.3}  &  -0.2     &  6.4420e-05  & 1.2431e-05  & 2.6710e-06  & 6.1586e-07  & 1.4766e-07   \\
                                &                       &           &              & 2.3735      & 2.2185      & 2.1167      & 2.0603       \\
                                &                       &  -0.8     &  1.6132e-03  & 4.2435e-04  & 1.0992e-04  & 2.8213e-05  & 7.2033e-06   \\
                                &                       &           &              & 1.9266      & 1.9487      & 1.9621      & 1.9696       \\  \cline{2-8}
                                & \multirow{4}{*}{0.7}  &  -0.2     &  2.8145e-04  & 6.7873e-05  & 1.6649e-05  & 4.1218e-06  & 1.0253e-06   \\
                                &                       &           &              & 2.0520      & 2.0274      & 2.0141      & 2.0072       \\
                                &                       &  -0.8     &  6.3566e-04  & 1.7068e-04  & 4.4407e-05  & 1.1358e-05  & 2.8782e-06   \\
                                &                       &           &              & 1.8969      & 1.9425      & 1.9671      & 1.9805       \\ \hline
\end{tabular}\label{table:Ex51}
\end{center}}
\end{table}


For  subdiffusion PDEs model \eqref{fee}, it is natural appearing the low regularity/singular  term  such as
$$t^{\mu} f(x,t)~~{\rm or}~~t^{\mu}\ast f(x,t), ~~\mu>-1.$$
In this case, many popular time stepping schemes, including  the correction of high-order  BDF methods  may lose their high-order accuracy,
see \cite[Section 4.1]{JLZ:2017} and Lemma 3.2 in \cite{WZ:2020}, also see Table \ref{table:6.1}.
The correction   BDF2 methods recovers superlinear convergence order $\mathcal{O}(\tau^{1+\alpha\mu})$, provided that the source term behaves like $t^{\alpha\mu}$,
which is invalid  for $\mu<0$, since it is required the source function $g\in C([0,T];L^{2}(\Omega))$.

To fill in this gap,
the desired second-order convergence rate can be achieved by ID1-BDF2 with $\mu>0$ but it is still  likely to exhibit a  order reduction with $\mu<0$.
Furthermore,  ID2-BDF2 method has filled a gap with $-1<\mu<0$, see Tables \ref{table:6.1} and \ref{table:Ex13}.
Tables \ref{table:Ex51} shows that ID1-BDF2 recovers    second order convergence and this is in  agreement with the order of the convergence for $t^{\mu} \ast f(x,t), ~\mu>-1$.

\begin{remark}
For  Hadamard's finite-Part integral \cite[p. 233]{Diethelm:2010}
$$\int_0^ts^{\mu}ds=\frac{1}{1+\mu}t^{1+\mu},~\mu<-1,$$
of course the limit does not exist, and so Hadamard suggested simply to ignore the unbounded contribution.
In this case, we can similar provide 
\begin{equation*}
{\rm ID3-BDF2~ Method:}~ \partial^{\alpha}_{\tau} V^{n} - AV^{n}= \partial^{3}_{\tau}  \left( \frac{t^{3}_{n}}{6} Av + \mathbb{G}^{n} \right),~
\mathbb{G}=J^{3}g(x,t),
\end{equation*}
which also  recovers the  high-order accuracy
even for the hypersingul source term,  see Table \ref{table:6.4}.
\begin{table}[!ht]
{\small \begin{center}
\caption{The discrete $L^2$-norm $||u^{N}-u^{2N}||$ and convergent order  with $\beta=0$, $\alpha=0.7$.
Here $\circ$ denotes the dot product.}
\begin{tabular}{l r l l l l l l}
\hline
Scheme                          &  $\mu$    &  $N=50$      &  $N=100$    & $N=200$     & $N=400$     &  $N=800$      \\ \hline
    \multirow{1}{*}{ID2-BDF2}   &  -1.8     &  1.7275e-02  & 8.1527e-03  & 3.6909e-03  & 1.6393e-03  & 7.2110e-04    \\
                                &           &              & 1.0834      & 1.1433      & 1.1709      & 1.1848        \\\hline
   \multirow{1}{*}{ID3-BDF2}    &  -1.8     &  7.7995e-03  & 1.8929e-03  & 4.6855e-04  & 9.5882e-05  & 2.2325e-05    \\
                                &           &              & 2.0428      & 2.0143      & 2.2889      & 2.1026        \\ \hline
\end{tabular}\label{table:6.4}
\end{center}}
\end{table}
\end{remark}

\section{Conclusions}
Fractional  PDEs model   naturally    imply  a less smooth or low regularity source function $t^{\mu}\circ f(x,t)$
 in the  right-hand side, which  is likely to result in a severe order reduction in most existing
time-stepping schemes. To fill in this gap, we provides a new idea to obtain the second-order time-stepping schemes for subdiffusion, called IDk-BDF2 method.
The detailed theoretical analysis and numerical verifications are presented.
In the future studies, we will try to adapt the idea to higher order schemes and the nonlinear fractional models \cite{Li:2022}.

\end{document}